\documentclass{amsart}

\usepackage{amsthm}
\usepackage{amssymb}
\usepackage{amsmath}
\usepackage{amsfonts}
\usepackage{amsrefs}
\usepackage{textcomp}
\usepackage[T1]{fontenc}
\usepackage[utf8x]{inputenc}
\usepackage{textcomp}
\usepackage{wasysym}
\usepackage{stmaryrd}
\usepackage{esint}
\usepackage[all]{xy}
\usepackage{graphicx}
\usepackage{bbm}
\usepackage{hyperref}

\newtheorem{theorem}{Theorem}[section]
\newtheorem{definition}[theorem]{Definition}
\newtheorem{remark}[theorem]{Remark}
\newtheorem{example}[theorem]{Example}
\newtheorem{lemma}[theorem]{Lemma}
\newtheorem{proposition}[theorem]{Proposition}
\newtheorem{corollary}[theorem]{Corollary}

\numberwithin{equation}{section}

\def\Xint#1{\mathchoice
   {\XXint\displaystyle\textstyle{#1}}%
   {\XXint\textstyle\scriptstyle{#1}}%
   {\XXint\scriptstyle\scriptscriptstyle{#1}}%
   {\XXint\scriptscriptstyle\scriptscriptstyle{#1}}%
   \!\int}
\def\XXint#1#2#3{{\setbox0=\hbox{$#1{#2#3}{\int}$}
     \vcenter{\hbox{$#2#3$}}\kern-.5\wd0}}

\def\dashint{\Xint-}

\newcommand{\twopartdef}[4]
{
\left\{
		\begin{array}{ll}
			#1 & #2 \\
			#3 & #4
		\end{array}
	\right.
}

\newcommand{\threepartdef}[6]
{
	\left\{
		\begin{array}{lll}
			#1 & #2 \\
			#3 & #4 \\
			#5 & #6
		\end{array}
	\right.
}

\usepackage{geometry}
\author{Wojciech G\'{o}rny}

\address{W. G\'{o}rny: Faculty of Mathematics, Informatics and Mechanics, University of Warsaw, Warsaw, Poland.}

\email{w.gorny@mimuw.edu.pl}

\date{\today}


\title[LGP with trace imposed on a part of the boundary]{Least gradient problem with Dirichlet condition imposed on a part of the boundary}


\begin{document}

\begin{abstract}
We provide an analysis of the least gradient problem in the case when the boundary datum is only imposed on a part of the boundary. First, we give a characterisation of solutions in a general setting using convex duality theory. Then, we discuss the way in which solutions attain their boundary values, structure of solutions and their regularity. 
\end{abstract}

\maketitle

\section{Introduction}

In this paper, we study a variant of the least gradient problem. In the last few years, this problem and its anisotropic formulation attracted a lot of attention, see for instance \cites{DS,GRS2017NA,JMN,HaM,HKLS,MRL,Mor,RS,Zun}. The standard version of the least gradient problem may be stated as follows:
\begin{equation}\label{LGP}\tag{LGP}
\min \bigg\{ \int_\Omega |Du|: \quad u \in BV(\Omega), \quad u|_{\partial\Omega} = f \bigg\}.
\end{equation}
The boundary datum is understood as the trace of a BV function. This problem was first considered in \cite{SWZ}, where the authors viewed it as primarily as a problem in geometric measure theory; it was studied under strict geometric conditions on $\Omega$ and the focus was on the relationship between problem \eqref{LGP} and the study of minimal surfaces (see also \cite{BGG}). The authors established that for continuous boundary data, if $\Omega \subset \mathbb{R}^N$ is an open bounded convex set, a unique solution exists and it is continuous up to the boundary. 

The main focus of this paper is the following variant of problem \eqref{LGP}:
\begin{equation}\label{problem}\tag{$\Gamma$-LGP}
\min \bigg\{ \int_\Omega |Du|_\phi: \quad u \in BV(\Omega), \quad u|_{\Gamma} = f \bigg\}.
\end{equation}
Here, $\Gamma$ is a relatively open subset of $\partial\Omega$, and the total variation is calculated with respect to the anisotropy given by the function $\phi$. The motivations to study such a problem are twofold. Firstly, on convex domains in two dimensions the problem \eqref{problem} (when $\phi$ is the Euclidean norm) is related to the problem appearing in {\it free material design}, see \cites{CL,KoZ}:
\begin{equation*}
\min \bigg\{ \int_{\overline{\Omega}} |p|: \quad p \in \mathcal{M}(\overline{\Omega}, \mathbb{R}^2), \quad \mathrm{div}(p) = 0,  \quad p \cdot \nu^\Omega|_{\Gamma} = g \bigg\},
\end{equation*}
where $g = \frac{\partial f}{\partial \tau}$ is the tangential derivative of $f$, see \cite{GRS2017NA}. The problem, first considered in \cite{GRS2017NA} in the planar and isotropic case, originates from mechanics; given a domain $\Omega$ and loads on the boundary (typically point loads), the goal is to find an elastic body which can support these loads and is as stiff as possible. Moreover, in this context it is natural to consider anisotropic norms, since in topology optimisation problems they correspond to using composite materials with anisotropic properties, see \cite{BeS}. An analogous problem in plastic design and its relationship to a constrained version of the least gradient problem was first studied by Kohn and Strang in \cite{KS}. Finally, let us note that when $\Gamma = \partial\Omega$, the free material design problem is also called the Beckmann problem and it is equivalent to the optimal transport problem, where the source and target measures are located on $\partial\Omega$:
\begin{equation*}
\min \bigg\{ \int_{\overline{\Omega} \times \overline{\Omega}} d\gamma: \, \gamma \in \mathcal{M}^+(\overline{\Omega}  \times \overline{\Omega}), \, (\Pi_x)_{\#}\gamma = g^+, \, (\Pi_y)_{\#} \gamma = g^- \bigg\},
\end{equation*}
where $g^+$ and $g^-$ are the positive and negative parts of $g$ respectively, see \cite{San2015}. 

The main motivation to consider anisotropic cases of the least gradient problem comes from medical imaging. Such a problem including a positive weight arises as a dimensional reduction of the conductivity imaging problem, see for instance \cite{JMN}. It is an inverse problem, where given a body $\Omega$, a measurement of the voltage $f$ on its boundary and a measurement of the current density $|J|$ inside the body we want to recover the (isotropic) conductivity $\sigma$. Denote by $u$ the electrical potential corresponding to the voltage $f$; then, it formally satisfies the equation
\begin{equation*}
\twopartdef{-\mathrm{div}(\sigma \nabla u) = 0}{\mbox{in } \Omega}{u = f}{\mbox{on } \partial\Omega.}
\end{equation*}
Because by Ohm's law the current density equals $J = - \sigma \nabla u$, the above equation can be formally rewritten as the weighted 1-Laplace equation
\begin{equation*}
\twopartdef{-\mathrm{div}(|J| \frac{\nabla u}{|\nabla u|}) = 0}{\mbox{in } \Omega}{u = f}{\mbox{on } \partial\Omega.}
\end{equation*}
The 1-Laplace equation is the Euler-Lagrange equation for the least gradient problem and this relationship extends to anisotropic cases (see \cites{Maz,MRL}), so the above equation is formally equivalent to the weighted least gradient problem with weight $a = |J|$:
\begin{equation}\label{eq:weightedleastgradientproblem}\tag{wLGP}
\min \bigg\{ \int_\Omega a(x) |Du|, \quad u \in BV(\Omega), \quad u|_{\partial\Omega} = f \bigg\}.
\end{equation}
The passage from the conductivity imaging problem to the weighted least gradient problem was presented here only on a formal level, but it was justified for $u \in W^{1,1}(\Omega) \cap C(\overline{\Omega})$ in \cite{MNT} and later for $u \in BV(\Omega)$ in \cite{MNT}.

This paper has three main objectives. The first one is to study a relaxed version of problem \eqref{problem}. The need to introduce a relaxed version can already be seen in the isotropic least gradient problem \eqref{LGP}, because even when $\Omega$ is a two-dimensional disk, there exist boundary data $f \in L^\infty(\partial\Omega)$ such that the least gradient problem (with boundary condition understood as the trace of a BV function) admits no solutions. We introduce the relaxed problem, which in particular involves a weaker form of the boundary condition, prove existence of minimisers and provide an Euler-Lagrange type characterisation of the solutions inspired by the characterisation given for the isotropic least gradient problem by Maz\'on, Rossi and Segura de Le\'on in \cite{MRL}. The method used in \cite{MRL} involved approximations by solutions to the $p$-Laplace equation as $p \rightarrow 1$; here, we use a different (and perhaps easier to generalise) method based on convex duality. Moreover, we prove that all solutions share the same frame of superlevel sets. This is done in Section \ref{sec:relaxedformulation}; the main result is Theorem \ref{thm:characterisation}.

The second goal is to study in more detail the way in which the boundary datum is attained. This part is inspired by the results of Jerrard, Moradifam and Nachman (\cite{JMN}). There, in the case when $\Gamma = \partial\Omega$, the authors prove that under a geometric assumption on $\Omega$ called the {\it barrier condition}, which is a generalisation of strict convexity to anisotropic cases, minimisers of the relaxed problem are minimisers of the original problem \eqref{problem}.
Here, we give a generalisation of this condition in Definition \ref{def:barrier}, and use it to recover the same implication when $\Gamma \neq \partial\Omega$. This is done in the first part of Section \ref{sec:strongformulation}; the main result is Theorem \ref{thm:existencetracesense}.

The third and final goal is to study regularity and structure of solutions. When $\Gamma = \partial\Omega$, in the isotropic case or when $\phi$ is regular enough (various sufficient conditions have been given in \cite{JMN} and \cite{Zun}), solutions to problem \eqref{problem} with continuous boundary data are continuous in $\overline{\Omega}$. Moreover, H\"older continuity of boundary data implies H\"older continuity of solutions with a smaller exponent. The situation is different when $\Gamma \neq \partial\Omega$. Then, it is natural for discontinuities to form: even in the isotropic case, solutions for continuous boundary data are not necessarily continuous in $\overline{\Omega}$. We show this in an extended series of examples which highlight different ways in which regularity of solutions may break down. 
It turns out that in order to have continuity of solutions inside $\Omega$, we need to assume that $\Gamma$ is connected, and even under this assumption we cannot hope for more than continuity of solutions in $\Omega \cup \Gamma$. Such a result for regularity of solutions is proved under the assumption that we have a maximum principle for $\phi$-minimal surfaces, which holds for instance if the anisotropy is given by a sufficiently regular weight or by a strictly convex norm in two dimensions. A similar discussion to the above is given to uniqueness and structure of solutions. This is done in the second part of Section \ref{sec:strongformulation}; the main results are Theorems \ref{thm:continuitynorms} and \ref{thm:continuityweights}.

\section{Preliminaries}\label{sec:preliminaries}

\subsection{Anisotropic BV spaces}

We start by recalling the notion of anisotropic BV spaces introduced in \cite{AB} and listing a few of their properties. A particular attention is given to properties involving traces of BV functions. There are various notations for the trace of a function $u \in BV(\Omega)$ on $\partial\Omega$ in the literature, such as $Tu$ or $\gamma u$, but in the whole paper we will simply denote it by $u$. Whenever there may be confusion we give an additional comment specifying if we mean a function (in $BV(\Omega)$) or its trace (in $L^1(\partial\Omega)$).

\begin{definition}
Let $\Omega \subset \mathbb{R}^N$ be an open bounded set with Lipschitz boundary. A continuous function $\phi: \overline{\Omega} \times \mathbb{R}^N \rightarrow [0, \infty)$ is called a metric integrand, if it satisfies the following conditions:  \\
$(1)$ $\phi$ is convex with respect to the second variable for a.e. $x \in \overline{\Omega}$; \\
$(2)$ $\phi$ is {1-}homogeneous with respect to the second variable, i.e.
\begin{equation*}
\forall \, x \in \overline{\Omega}, \quad \forall \, \xi \in \mathbb{R}^N, \quad \forall \, t \in \mathbb{R} \quad \phi(x, t \xi) = |t| \phi(x, \xi);
\end{equation*}
$(3)$ $\phi$ is comparable to the Euclidean norm on $\overline{\Omega}$, i.e.
\begin{equation*}
\exists \,  \lambda, \Lambda  > 0 \quad \forall \, x \in \overline{\Omega}, \quad \forall \, \xi \in \mathbb{R}^N \quad \lambda |\xi| \leq \phi(x, \xi) \leq \Lambda |\xi|.
\end{equation*}
In particular, $\phi$ is uniformly elliptic in $\overline{\Omega}$.
\end{definition}

In the context of least gradient problems, these conditions apply to most cases considered in the literature. The typical forms of $\phi$ include: $\phi(x, \xi) = |\xi|$ (the classical least gradient problem, see \cites{GRS2017NA,MRL,SWZ}); $\phi(x, \xi) = g(x) |\xi|$ with $g$ continuous and bounded away from $0$ (the weighted least gradient problem, see \cites{JMN,Zun}); $\phi(x, \xi) = \| \xi \|_p$, where $p \in [1, \infty]$ $($anisotropy defined by the $l_p$ norms, see \cites{Gor2018CVPDE,Gor2020NA}$)$.

\begin{definition}
The polar function of $\phi$ is $\phi^0: \overline{\Omega} \times \mathbb{R}^N \rightarrow [0, \infty)$ defined by the formula
\begin{equation*}
\phi^0 (x, \xi^*) = \sup \, \{ \langle \xi^*, \xi \rangle : \, \xi \in \mathbb{R}^N, \, \phi(x, \xi) \leq 1 \}.
\end{equation*} 
\end{definition}

\begin{definition}
Let $\phi$ be a {continuous metric integrand in} $\overline{\Omega}$. For a given function $u \in L^1(\Omega)$ we define its $\phi-$total variation in $\Omega$ by the formula:
\begin{equation*}
\int_\Omega |Du|_\phi = \sup \, \bigg\{ \int_\Omega u \, \mathrm{div}(\mathbf{z}) \, dx : \, \phi^0(x,\mathbf{z}(x)) \leq 1 \, \, \, \text{a.e.}, \quad \mathbf{z} \in C_c^1(\Omega) \bigg\}.
\end{equation*}
The $\phi-$total variation is also sometimes denoted $\int_\Omega \phi(x, Du)$. We will say that $u \in BV_\phi(\Omega)$ if its $\phi-$total variation in $\Omega$ is finite; furthermore, we define the $\phi-$perimeter of a set $E$ by the formula
$$P_\phi(E, \Omega) = \int_{\Omega} |D\chi_E|_\phi.$$
If $P_\phi(E, \Omega) < \infty$, we say that $E$ is a set of bounded $\phi-$perimeter in $\Omega$.
\end{definition}

This definition is very similar to the definition of standard BV spaces; the only difference is that the bound on the length of the vector field is expressed in terms of the polar norm of $\phi$. The properties of metric integrands ensure that $BV(\Omega) = BV_\phi(\Omega)$ as sets, equipped with different (but equivalent) topologies, and that many properties of isotropic BV spaces can be recovered. We are primarily concerned with approximation by smooth functions and a version of the Gagliardo extension theorem; in the form presented below they were proved in \cite{Moll}.

\begin{lemma}\label{lem:anisotropicstrictapproximation}
Given $u \in BV(\Omega)$, there exists a sequence $w_n \in W^{1,1}(\Omega)$ such that $w_n \rightarrow u$ in $L^1(\Omega)$, $w_n = f$ on $\partial\Omega$ and
\begin{equation}
\int_\Omega |Du|_\phi = \lim_{n \rightarrow \infty} \int_\Omega \phi(x, \nabla w_n(x)) \, dx.
\end{equation}
\end{lemma}

\begin{lemma}\label{lem:anisotropicgagliardo}
Given $g \in L^1(\partial\Omega)$, there exists a sequence $v_n \in W^{1,1}(\Omega)$ such that $v_n = g$ on $\partial\Omega$, $v_n(x) = 0$ if $\mbox{dist}(x, \partial\Omega) > \frac{1}{n}$ and
\begin{equation}
\int_\Omega |Dv_n|_\phi \leq \int_{\partial\Omega} \phi(x,\nu^\Omega) \, |g| \, d\mathcal{H}^{N-1} + \frac{1}{n}.
\end{equation}
\end{lemma}

\subsection{Anzelotti pairings}

Now, we recall the definition and basic properties of Anzelotti pairings introduced in \cite{Anz}; we follow the presentation of this subject in \cite{CFM} (in the isotropic case a good introduction can be found in Appendix C to \cite{ACM}). Suppose that $\Omega \subset \mathbb{R}^N$ is an open bounded set with Lipschitz boundary. For $p \geq 1$, denote
\begin{equation*}
X_p(\Omega) = \bigg\{ \mathbf{z} \in L^\infty(\Omega; \mathbb{R}^N): \, \mbox{ div}(\mathbf{z}) \in L^p(\Omega) \bigg\}.
\end{equation*}
Given $\mathbf{z} \in X_N(\Omega)$ and $w \in BV(\Omega) \subset L^{N/(N-1)}(\Omega)$, we define the functional $(\mathbf{z}, Dw): C_c^\infty(\Omega) \rightarrow \mathbb{R}$ by the formula
\begin{equation*}
\langle (\mathbf{z}, Dw), \varphi \rangle = - \int_\Omega w \, \varphi \, \mathrm{div}(\mathbf{z}) \, dx - \int_\Omega w \, \mathbf{z} \cdot \nabla \varphi \, dx.
\end{equation*}
The distribution $(\mathbf{z}, Dw)$ turns out to be a Radon measure on $\Omega$. It generalises the pointwise product $\mathbf{z} \cdot \nabla w$ to $BV(\Omega)$, namely for $w \in W^{1,1}(\Omega) \cap L^\infty(\Omega)$ we have
\begin{equation*}
\int_\Omega (\mathbf{z}, Dw) = \int_{\Omega} \mathbf{z} \cdot \nabla w \, dx \quad \forall \, w \in W^{1,1}(\Omega).
\end{equation*}
The following Proposition summarises the most important properties of the pairing $(\mathbf{z}, Du)$.

\begin{proposition}\label{prop:boundonanzelottipairing}
Suppose that $\Omega \subset \mathbb{R}^N$ is an open bounded set with Lipschitz boundary. Suppose that $\phi$ is a metric integrand. Let $\mathbf{z} \in X_N(\Omega)$ and $u \in BV(\Omega)$. Then, for any Borel set $B \subset \Omega$ we have
\begin{equation*}
\bigg| \int_{B} (\mathbf{z}, Du) \bigg| \leq \| \phi^0(x,\mathbf{z}(x)) \|_{L^\infty(\Omega)} \int_B |Du|_\phi,
\end{equation*}
in particular $(\mathbf{z}, Du) \ll |Du|$ as measures in $\Omega$.

Moreover, there exists a function $[\mathbf{z},\nu^\Omega] \in L^\infty(\partial\Omega)$ such that $\| [\mathbf{z},\nu^\Omega] \|_{L^\infty(\partial\Omega)} \leq  \| \mathbf{z} \|_{L^\infty(\Omega; \mathbb{R}^N)}$ and the following {\it Green's formula} holds:
\begin{equation*}
\int_\Omega u \, \mathrm{div}(\mathbf{z}) \, dx + \int_\Omega (\mathbf{z}, Du) = \int_{\partial\Omega} [\mathbf{z}, \nu^\Omega] \, u \, d\mathcal{H}^{N-1}.
\end{equation*}
\end{proposition}

The function $[\mathbf{z},\nu^\Omega]$ has the interpretation of the normal trace of the vector field $\mathbf{z}$ at the boundary and it coincides with the classical normal trace if $\mathbf{z}$ is smooth enough. Moreover, the construction above can be done under slightly more general assumptions (see \cites{Anz,CFM}), but here we restrict ourselves to the setting we will use in Section \ref{sec:relaxedformulation}.

\subsection{Anisotropic least gradient functions}

\begin{definition}
Let $\Omega \subset \mathbb{R}^N$ be an open bounded set with Lipschitz boundary. We say that $u \in BV(\Omega)$ is a function of $\phi-$least gradient $($in $\Omega)$, if for every compactly supported $v \in BV(\Omega)$ we have
\begin{equation*}
\int_\Omega |Du|_\phi \leq \int_\Omega |D(u + v)|_\phi.
\end{equation*}
If $\phi$ admits a continuous extension to $\mathbb{R}^N$, we may instead assume that $v$ is a $BV$ function with zero trace on $\partial\Omega$; see \cite[Proposition 3.16]{Maz}.
\end{definition}

Additionally, if a set $E \subset \Omega$ is such that $\chi_E$ is a function of $\phi-$least gradient, we say that $E$ is a $\phi-$minimal set.

\begin{definition}
We say that $u \in BV(\Omega)$ is a solution to Problem \eqref{problem}, if $u$ is a function of $\phi-$least gradient and the trace of $u$ on $\Gamma$ equals $f$, i.e. for $\mathcal{H}^{N-1}-$almost every $x \in \Gamma$ we have 
$$ {\lim_{r \rightarrow 0^+} \,} \dashint_{B(x,r) \cap \Omega} |f(x) - u(y)| \, dy = 0.$$
\end{definition}

However, this condition is a very strong notion of solutions even in the case when $\Gamma = \partial\Omega$. Typically, in the least gradient problem, solutions in this sense exist only for regular enough boundary data and under additional geometric conditions on $\Omega$ (see \cites{Gor2020NA,Gor2019IUMJ,JMN,Mor,RS,ST,SWZ}); in the isotropic case, a sufficient condition is strict convexity of $\Omega$ and continuity of $f$, see \cite{SWZ}. In Section \ref{sec:relaxedformulation}, we will introduce a different notion of solutions, see Definition \ref{dfn:1laplace}, and in Section \ref{sec:strongformulation} we will discuss the relationship between the two definitions under additional assumptions on $f$ and $\Omega$.

Finally, let us mention a characterisation of $\phi$-least gradient functions via their superlevel sets. The first result of this type has been proved in the isotropic case in \cite[Theorem 1]{BGG} and its proof is based on the the co-area formula.

\begin{theorem}\label{thm:anisobgg}
$($\cite[Theorem 3.19]{Maz}$)$ Let $\Omega \subset \mathbb{R}^N$ be an open bounded set with Lipschitz boundary. Assume that $\phi$ admits a continuous extension to $\mathbb{R}^N$. Take $u \in BV(\Omega)$. Then, $u$ is a function of $\phi-$least gradient in $\Omega$ if and only if $\chi_{E_t}$ is a function of $\phi-$least gradient in $\Omega$ for almost all $t \in \mathbb{R}$.  
\end{theorem}

\section{Relaxed formulation of the problem}\label{sec:relaxedformulation}

In this Section, we will impose the Dirichlet boundary condition only on a part of the boundary. Namely, let us take $\Gamma \subset \partial\Omega$ to be a relatively open subset. Given $f \in L^1(\partial\Omega)$, we consider the functional $J_\Gamma: L^{1}(\Omega) \rightarrow [0,\infty]$ defined by the formula
\begin{equation}
J_\Gamma(u) = \twopartdef{\int_\Omega |Du|_\phi}{\mbox{if } u \in BV(\Omega), \,\, u = f \mbox{ on } \Gamma}{+\infty}{\mbox{otherwise.}}
\end{equation}
Study of the problem \eqref{problem} corresponds to minimisation of this functional in $L^1(\Omega)$. However, even when $\Gamma = \partial\Omega$, this minimisation procedure faces some geometric difficulties. In this case, existence of solutions has been proved for continuous boundary data under some additional geometric assumptions on $\Omega$, such as positive mean curvature of $\partial\Omega$ in the isotropic case (see \cite{SWZ}) or the {\it barrier condition} in the anisotropic case (see \cite{JMN}). Furthermore, if the boundary data are discontinuous, it is possible that there are no solutions even in the isotropic case when and $\Omega$ is a disk, see \cite{ST}.

For these reasons, it is natural to study the relaxed functional of $J_\Gamma$, namely the functional $\overline{J}_\Gamma: L^1(\Omega) \rightarrow [0,\infty]$ defined by
\begin{equation}
\overline{J}_\Gamma(u) = \inf \bigg\{ \liminf_{n \rightarrow \infty} J_\Gamma(u_n): \quad u_n \rightarrow u \mbox{ in } L^1(\Omega), \quad u_n \in BV(\Omega), \quad u_n = f \mbox{ on } \Gamma  \bigg\}.
\end{equation}
We will see that if $\Gamma \subset \partial\Omega$ is regular enough, then we may give an exact formula for $\overline{J}_\Gamma$. Consider the functional
$\mathcal{J}_\Gamma: L^{1}(\Omega) \rightarrow [0,\infty]$ defined by the formula
\begin{equation}
\mathcal{J}_{\Gamma}(u) = \int_{\Omega} |Du|_\phi + \int_{\Gamma} \phi(x,\nu^\Omega) |u - f| \, d\mathcal{H}^{N-1}.
\end{equation}
Under a certain geometric assumption on $\Gamma$, we will see in Theorem \ref{thm:relaxation} that $\overline{J}_\Gamma = \mathcal{J}_\Gamma$.

\subsection{Relaxation of the functional}

This subsection is devoted to the study of the relaxed functional of $J_\Gamma$. The analysis will be performed under the following geometric assumption on $\Gamma$:

\begin{definition}\label{dfn:lipschitzextension}
Suppose that $\Omega \subset \mathbb{R}^N$ be an open bounded set with Lipschitz boundary. Let $\Gamma \subset \partial\Omega$. We say that $\Omega$ satisfies the Lipschitz extension property near $\Gamma$, if there exists an open bounded set $\Omega'$ with Lipschitz boundary such that $\Omega \subset \Omega'$ and
$$ \partial\Omega \cap \partial\Omega' = \partial\Omega \backslash \Gamma.$$
\end{definition}

This is in fact a regularity assumption on $\Gamma$. It is satisfied in a variety of cases, for instance when $\Omega$ is a strictly convex set on the plane and $\Gamma$ is a finite union of arcs, or in any dimension when $\Gamma$ is the part of the boundary cut off by a hyperplane.

\begin{proposition}\label{prop:lowersemicontinuity}
Suppose that $\Omega$ satisfies the Lipschitz extension property near $\Gamma$. Then, the functional $\mathcal{J}_\Gamma$ is lower semicontinuous on $L^{1}(\Omega)$.
\end{proposition}

\begin{proof}
Let $\Omega'$ be the set in Definition \ref{dfn:lipschitzextension}. Let $\psi \in W^{1,1}(\Omega' \backslash \overline{\Omega})$ be a function with trace $f$ on $\Gamma$. Denote by $u_\psi \in BV(\Omega')$ the function defined by
\begin{equation}
u_\psi(x) = \twopartdef{u(x)}{x \in \Omega}{\psi(x)}{x \in \Omega' \backslash \overline{\Omega}.}
\end{equation}
Then, we have (see for instance \cite[Corollary 3.89]{AFP} in the isotropic case)
$$ \int_{\Omega'} |Du_\psi|_\phi = \int_\Omega |Du|_\phi + \int_{\Gamma} \phi(x, \nu^\Omega) |u - f| \, d\mathcal{H}^{N-1} + \int_{\Omega' \backslash \Omega} \phi(x,\nabla \psi(x)) \, dx.$$
We rewrite the above as follows:
$$ \mathcal{J}_{\Gamma}(u) = \int_{\Omega} |Du|_\phi + \int_{\Gamma} \phi(x,\nu^\Omega) |u - f| \, d\mathcal{H}^{N-1} = \int_{\Omega'} |Du_\psi|_\phi - \int_{\Omega' \backslash \Omega} \phi(x,\nabla \psi(x)) \, dx.$$
Now, suppose that $u_n \rightarrow u$ in $L^{1}(\Omega)$. In particular, also $(u_n)_\psi \rightarrow u_\psi$ in $L^1(\Omega')$. Then, by the lower semicontinuity of the $\phi-$total variation,
$$ \liminf_{n \rightarrow \infty} \mathcal{J}_{\Gamma}(u_n) = \liminf_{n \rightarrow \infty}  \int_{\Omega'} |D(u_n)_\psi|_\phi - \int_{\Omega' \backslash \Omega} \phi(x,\nabla \psi(x)) \, dx \geq \qquad\qquad\qquad\qquad\qquad\qquad\qquad$$
$$ \qquad\qquad\qquad\qquad\qquad\qquad\qquad\qquad\qquad\qquad \geq \int_{\Omega'} |Du_\psi|_\phi - \int_{\Omega' \backslash \Omega} \phi(x,\nabla \psi(x)) \, dx = \mathcal{J}_\Gamma(u),$$
so the functional $\mathcal{J}_\Gamma$ is lower semicontinuous on $L^{1}(\Omega)$.
\end{proof}

\begin{proposition}\label{prop:approximation}
Suppose that $\Omega$ satisfies the Lipschitz extension property near $\Gamma$. Given $u \in BV(\Omega)$, there exists a sequence $u_n \in W^{1,1}(\Omega)$ such that $u_n \rightarrow u$ in $L^1(\Omega)$, $u_n = f$ on $\Gamma$ and
\begin{equation}
\mathcal{J}_{\Gamma}(u) = \lim_{n \rightarrow \infty} J_\Gamma(u_n).
\end{equation}
\end{proposition}

\begin{proof}
We set
\begin{equation}
g = \twopartdef{f - u}{\mbox{on } \Gamma}{0}{\mbox{on } \partial\Omega \backslash \Gamma.}
\end{equation}
Let $w_n$ be the sequence given by Lemma \ref{lem:anisotropicstrictapproximation} and let $v_n$ be the sequence given by Lemma \ref{lem:anisotropicgagliardo}. We have $w_n \rightarrow u$ in $L^1(\Omega)$, $v_n \rightarrow 0$ in $L^1(\Omega)$ and $v_n = g$ on $\partial\Omega$. Moreover, we rewrite the estimate in Lemma \ref{lem:anisotropicgagliardo} as
\begin{equation}
\int_\Omega |Dv_n|_\phi \leq \int_{\Gamma} \phi(x,\nu^\Omega) \, |u - f| \, d\mathcal{H}^{N-1} + \frac{1}{n}.
\end{equation}
Set $u_n = v_n + w_n$. Then, $u_n \in W^{1,1}(\Omega)$, $u_n \rightarrow u$ in $L^1(\Omega)$ and $u_n = f$ on $\Gamma$. We estimate
\begin{equation}
\mathcal{J}_\Gamma(u_n) = \int_\Omega |Du_n|_\phi \leq \int_\Omega |Dv_n|_\phi + \int_\Omega |Dw_n|_\phi \leq \int_\Gamma \phi(x,\nu^\Omega) \, |u - f| \, d\mathcal{H}^{N-1} + \frac{1}{n} + \int_\Omega |Dw_n|_\phi.
\end{equation}
Now, we take the upper limit in the above series of inequalities. By the lower semicontinuity of $\mathcal{J}_\Gamma$ given in Proposition \ref{prop:lowersemicontinuity}, we get
\begin{equation*}
\mathcal{J}_\Gamma(u) \leq \liminf_{n \rightarrow \infty} \mathcal{J}_\Gamma(u_n) \leq \limsup_{n \rightarrow \infty} \mathcal{J}_\Gamma(u_n) \leq \limsup_{n \rightarrow \infty} \int_\Omega |Dw_n|_\phi + \int_{\Gamma} \phi(x,\nu) \, |u - f| \, d\mathcal{H}^{N-1} = \qquad
\end{equation*}
\begin{equation*}
\qquad\qquad\qquad\qquad\qquad\qquad\qquad\qquad\qquad = \lim_{n \rightarrow \infty} \int_\Omega |Dw_n|_\phi + \int_{\Gamma} \phi(x,\nu^\Omega) \, |u - f| \, d\mathcal{H}^{N-1} = \mathcal{J}_\Gamma(u).
\end{equation*}
Hence, all the inequalities above are in fact equalities and $u_n$ satisfies all the desired properties.
\end{proof}

Finally, notice that Propositions \ref{prop:lowersemicontinuity} and \ref{prop:approximation} immediately imply the following Theorem.

\begin{theorem}\label{thm:relaxation}
Suppose that $\Omega$ satisfies the Lipschitz extension property near $\Gamma$. Then, the relaxation of the functional $J_\Gamma$ is the functional $\mathcal{J}_\Gamma$. \qed
\end{theorem}

Therefore, in what follows we will study the properties of the functional $\mathcal{J}_\Gamma$.

\subsection{The Euler-Lagrange characterisation}

In this Section, we want to study existence of solutions to \eqref{problem} in the sense of Euler-Lagrange equations. This is motivated by the following observation: in the isotropic least gradient problem with the Dirichlet boundary condition imposed on the whole boundary, the Euler-Lagrange equation of
\begin{equation*}
\min \bigg\{ \int_\Omega |Du|: \quad u \in BV(\Omega), \quad u = f \mbox{ on } \partial\Omega   \bigg\}
\end{equation*}
is formally given by the $1-$Laplace equation
\begin{equation*}
\twopartdef{-\mathrm{div}\bigg( {\displaystyle \frac{Du}{|Du|} }\bigg) = 0}{\mbox{ in } \Omega}{u = f}{\mbox{ on } \partial\Omega.}
\end{equation*}
This formal expression was first given a precise meaning by Maz\'on, Rossi and Segura de Le\'on in \cite{MRL}. The authors provide a characterisation of solutions to the $1-$Laplace equation by introducing a divergence-free vector field $\mathbf{z}$ which plays a role of the expression $\frac{Du}{|Du|}$ even when it is not well-defined. The vector field $\mathbf{z}$ is obtained using an approximation by solutions to the Dirichlet problems for the $p-$Laplace equations as $p \rightarrow 1$. A similar idea appears in \cite{Maz}, where the author used the Yosida approximation of the subdifferential in order to recover an analogous result in the anisotropic case. 

Here, we use a different (and perhaps simpler) approach. Instead of solving a sequence of approximate problems, we will use the duality theory in the sense of Ekeland-Temam (\cite{ET}). We restrict the domain of the functional $\mathcal{J}_\Gamma$ to $W^{1,1}(\Omega)$, and find the dual problem to the minimisation of $\mathcal{J}_\Gamma$ in $W^{1,1}(\Omega)$. We will see that the dual problem admits a solution (even if the primal problem does not) and we will use the $\varepsilon-$subdifferentiability property to obtain the Euler-Lagrange equations for any minimiser of the functional $\mathcal{J}_\Gamma$ in $BV(\Omega)$.

First, let us recall that $BV(\Omega) \subset L^{N/(N-1)}(\Omega)$. The dual space to $L^{N/(N-1)}(\Omega)$ is $L^N(\Omega)$; in this duality, for any $u \in L^{N/(N-1)}$ we can define the subdifferential of the convex and lower semicontinuous (provided that $\Omega$ satisfies the Lipschitz extension property near $\Gamma$) functional $\mathcal{J}_\Gamma$ as follows:
\begin{equation*}
\partial \mathcal{J}_\Gamma (u) := \bigg\{ w \in L^N(\Omega): \, \mathcal{J}_\Gamma(v) - \mathcal{J}_\Gamma(u) \geq \int_\Omega w (v - u) \, dx \quad  \forall v \in L^{N/(N-1)}(\Omega) \bigg\}.
\end{equation*}
Under these assumptions, the subdifferential $\partial\mathcal{J}_\Gamma(u)$ is a convex, closed and nonempty set. Moreover, $u$ is a minimiser of the functional $\mathcal{J}_\Gamma$ if and only if $0 \in \partial\mathcal{J}_\Gamma(u)$. Therefore, the Euler-Lagrange equation associated to minimisation of $\mathcal{J}_\Gamma$ is
\begin{equation}\label{eq:variationalsolution}
0 \in \partial\mathcal{J}_\Gamma(u)
\end{equation}
(note that this incorporates the Dirichlet boundary condition imposed on $\Gamma$). Following \cite{MRL}, we now give a precise characterisation of solutions to \eqref{eq:variationalsolution}:

\begin{definition}\label{dfn:1laplace}
We will say that $u \in BV(\Omega)$ is a solution to the anisotropic $1-$Laplace equation with Dirichlet boundary datum on $\Gamma \subset \partial\Omega$, if there exists a vector field $\mathbf{z} \in L^\infty(\Omega; \mathbb{R}^N)$ such that $\phi^0(x,\mathbf{z}(x)) \leq 1$ a.e. in $\Omega$ which satisfies:
\begin{equation}\label{eq:divergencecondition}
-\mathrm{div}(\mathbf{z}) = 0 \qquad \mbox{ in } \mathcal{D}'(\Omega);
\end{equation}
\begin{equation}\label{eq:directionalcondition}
(\mathbf{z}, Du) = |Du|_\phi \qquad \mbox{ as measures in } \Omega; 
\end{equation}
\begin{equation}
[\mathbf{z}, \nu^\Omega] \in \mbox{sign}(f-u) \, \phi(\cdot, \nu^\Omega) \qquad \mbox{ a.e. on } \Gamma;
\end{equation}
\begin{equation}
[\mathbf{z}, \nu^\Omega] = 0 \qquad \mbox{ a.e. on } \partial\Omega \backslash \Gamma.
\end{equation}
\end{definition}

The name ``anisotropic $1$-Laplace equation'' comes from the fact that in the isotropic case, assuming that $u \in W^{1,1}(\Omega)$ with $|\nabla u| > 0$, we have $\mathbf{z} = \frac{\nabla u}{|\nabla u|}$ a.e. in $\Omega$, so equations \eqref{eq:divergencecondition} and \eqref{eq:directionalcondition} reduce to
\begin{equation*}
-\mathrm{div}\bigg( {\displaystyle \frac{\nabla u}{|\nabla u|} }\bigg) = 0.
\end{equation*}
In the next subsection, we prove that the two notions of solutions given in equation \eqref{eq:variationalsolution} and Definition \ref{dfn:1laplace} are indeed equivalent. Then, we will study the relationship between these formulations and the anisotropic least gradient problem with Dirichlet boundary condition on a part of the boundary, i.e. problem \eqref{problem}.

\subsection{Proof of the Euler-Lagrange characterisation}

We want to prove that any minimiser $u \in BV(\Omega)$ of the functional $\mathcal{J}_\Gamma$ satisfies the conditions in Definition \ref{dfn:1laplace}. To this end, we will study the dual problem to the minimisation of $\mathcal{J}_\Gamma$. First, let us recall the notion of the Legendre-Fenchel transform. It is defined as follows: given a Banach space $V$ and $F: V \rightarrow \mathbb{R} \cup \{ + \infty \}$, we define $F^*: V^* \rightarrow \mathbb{R} \cup \{ + \infty \}$ by the formula
\begin{equation*}
F^*(v^*) = \sup_{v \in V} \bigg\{ \langle v, v^* \rangle - F(v) \bigg\}.
\end{equation*}
Then, let us recall shortly how the dual problem is typically defined in the setting of calculus of variations. A standard reference is \cite[Chapter III.4]{ET}. 

Let $X, Y$ be two Banach spaces and let $A: X \rightarrow Y$ be a continuous linear operator. Denote by $A^*: Y^* \rightarrow X^*$ its dual. Then, if the primal problem is of the form
\begin{equation}\tag{P}\label{eq:primal}
\inf_{u \in X} \bigg\{ E(Au) + G(u) \bigg\},
\end{equation}
where $E: Y \rightarrow \mathbb{R} \cup \{ +\infty \}$ and $G: X \rightarrow \mathbb{R} \cup \{ +\infty \}$ are proper, convex and lower semicontinuous, then the dual problem is defined as the maximisation problem
\begin{equation}\tag{P*}\label{eq:dual}
\sup_{p^* \in Y^*} \bigg\{ - E^*(-p^*) - G^*(A^* p^*) \bigg\}.
\end{equation}
Moreover, if there exists $u_0 \in X$ such that $E(A u_0) < \infty$, $G(u_0) < \infty$ and $E$ is continuous at $A u_0$, then
$$\inf \eqref{eq:primal} = \sup \eqref{eq:dual}$$
and the dual problem \eqref{eq:dual} admits at least one solution.

Let us express the minimisation of $\mathcal{J}_\Gamma$ in this framework. We restrict its domain of definition to $W^{1,1}(\Omega)$, so that the gradient is a bounded operator from $W^{1,1}(\Omega)$ to $(L^1(\Omega))^N$ and the dual spaces are easy to control. Namely, we minimise the functional $F: W^{1,1}(\Omega) \rightarrow [0,\infty]$ given by the same formula as $\mathcal{J}_\Gamma$, i.e.
\begin{equation}
F(u) = \int_{\Omega} \phi(x,\nabla u(x)) \, dx + \int_{\Gamma} \phi(x,\nu^\Omega) \, |u - f|  \, d\mathcal{H}^{N-1}.
\end{equation}
We want to express the minimisation of the functional $F$ in the framework of Fenchel duality. Therefore, we set $X = W^{1,1}(\Omega)$, $Y = L^1(\partial\Omega) \times (L^1(\Omega))^N$, and the linear operator $A: X \rightarrow Y$ is defined by the formula
\begin{equation*}
Au = (u|_{\partial\Omega}, \nabla u).
\end{equation*}
Here, $u|_{\partial\Omega}$ is the trace of $u \in W^{1,1}(\Omega)$ on $\partial\Omega$. In particular, the dual spaces to $X$ and $Y$ are
\begin{equation*}
X^* = (W^{1,1}(\Omega))^*, \qquad Y^* =  L^\infty(\partial\Omega) \times (L^\infty(\Omega))^N.
\end{equation*}
We denote the points $p \in Y$ in the following way: $p = (p_0, \overline{p})$, where $p_0 \in L^1(\partial\Omega)$ and $\overline{p} \in (L^1(\Omega))^N$. We will also use a similar notation for points $p^* \in Y^*$. Then, we set $E: L^1(\partial\Omega) \times (L^1(\Omega))^N \rightarrow \mathbb{R}$ by the formula
\begin{equation}\label{eq:definitionofE}
E(p_0, \overline{p}) = E_0(p_0) + E_1(\overline{p}), \quad E_0(p_0) = \int_{\Gamma} \phi(x,\nu^\Omega) |p_0 - f| \, d\mathcal{H}^{N-1}, \quad E_1(\overline{p}) = \int_\Omega \phi(x,\overline{p}(x)) \, dx.
\end{equation}
We also set $G: W^{1,1}(\Omega) \rightarrow \mathbb{R}$ to be the zero functional, i.e. $G \equiv 0$. In particular, the functional $G^*: (W^{1,1}(\Omega))^* \rightarrow [0,\infty]$ is given by the formula
\begin{equation*}
G^*(u^*) = \twopartdef{0}{\mbox{if } u^* = 0;}{+\infty}{\mbox{if } u^* \neq 0.}
\end{equation*}
By \cite[Lemma 2.1]{Mor} the functional $E_1^*: (L^\infty(\Omega))^N \rightarrow [0,\infty]$ is given by the formula
\begin{equation*}
E_1^*(\overline{p}^*) = \twopartdef{0}{\mbox{if } \phi^0(x,\overline{p}^*(x)) \leq 1 \mbox{ a.e. in } \Omega;}{+\infty}{\mbox{otherwise}.}
\end{equation*}
It remains to calculate the functional $E_0^*: L^\infty(\partial\Omega) \rightarrow \mathbb{R} \cup \{ \infty \}$. 

\begin{lemma}\label{lem:E0dual}
Let $E_0$ be defined in equation \eqref{eq:definitionofE}. Then, we have
\begin{equation}
E_0^*(p_0^*) = \twopartdef{\int_{\partial\Omega} f \, p_0^* \, d\mathcal{H}^{N-1}}{p_0^* = 0 \mbox{ a.e. on } \partial\Omega\backslash\Gamma, |p_0^*| \leq \phi(x,\nu^\Omega) \mbox{ a.e. on } \Gamma;}{+\infty}{\mbox{otherwise.}}    
\end{equation}
\end{lemma}

\begin{proof}
First, we will prove that
\begin{equation}\label{eq:calculatingthedual}
\sup_{p_0 \in L^1(\partial\Omega)} \bigg\{ \langle p_0, p_0^* \rangle - \int_{\Gamma} \phi(x,\nu^\Omega) \, |p_0| \, d\mathcal{H}^{N-1} \bigg\} \end{equation}
equals zero if $p_0^* = 0$ $\mbox{a.e. on } \partial\Omega \backslash \Gamma$ with respect to $\mathcal{H}^{N-1}$ and $|p_0^*| \leq \phi(x,\nu^\Omega)$ $\mbox{a.e. on } \Gamma$ with respect to $\mathcal{H}^{N-1}$. Otherwise, this value equals $+\infty$. 

Suppose otherwise. If $p_0^* \neq 0$ on a subset of $\partial\Omega \backslash \Gamma$ of positive $\mathcal{H}^{N-1}-$measure, then there exists a set $K \subset \Gamma$ of positive $\mathcal{H}^{N-1}-$measure such that $|p_0^*| > \varepsilon$ on $K$. Without loss of generality, assume that $p_0^* > \varepsilon$ on $K$. Take the sequence $p_0^k = k \chi_K$; then, we have
\begin{equation}
\langle p_0^k, p_0^* \rangle - \int_{\Gamma} \phi(x,\nu^\Omega) \, |p_0^k| \, d\mathcal{H}^{N-1} = \langle p_0^k, p_0^* \rangle \geq \int_K k \varepsilon \, d\mathcal{H}^{N-1} \rightarrow +\infty,
\end{equation}
so the expression in \eqref{eq:calculatingthedual} goes to $+\infty$ as $k \rightarrow \infty$.

Similarly, assume that $|p_0^*| > \phi(x,\nu^\Omega)$ on a subset of $\Gamma$ of positive $\mathcal{H}^{N-1}-$measure. Then, there exists a set $K' \subset \Gamma$ of positive $\mathcal{H}^{N-1}-$measure such that $|p_0^*| > \phi(x,\nu^\Omega) + \varepsilon$ on $K'$. Without loss of generality, assume that $p_0^* > \phi(x,\nu^\Omega) + \varepsilon$ on $K'$. Take the sequence $p_0^l = l \chi_K'$; then, we have
\begin{equation*}
\langle p_0^l, p_0^* \rangle - \int_{\Gamma} \phi(x,\nu^\Omega) \, |p_0^l| \, d\mathcal{H}^{N-1} \geq \int_{K'} l |p_0^*| \, d\mathcal{H}^{N-1} - \int_{K'} l \phi(x,\nu^\Omega) \, d\mathcal{H}^{N-1} \geq \int_{K'} l\varepsilon \, d\mathcal{H}^{N-1} \rightarrow +\infty,
\end{equation*}
so the expression in \eqref{eq:calculatingthedual} goes to $+\infty$ as $l \rightarrow \infty$.

Now, let us see that if these two conditions are satisfied, then the expression in \eqref{eq:calculatingthedual} is bounded from above by zero. Indeed, we have
\begin{equation*}
\langle p_0, p_0^* \rangle - \int_{\Gamma} \phi(x,\nu^\Omega) \, |p_0| \, d\mathcal{H}^{N-1} \leq \int_\Gamma \phi(x,\nu^\Omega) (p_0 - |p_0|) \, d\mathcal{H}^{N-1} \leq 0,
\end{equation*}
so we proved equation \eqref{eq:calculatingthedual}.

Finally, we compute
$$ E_0^*(p_0^*) = \sup_{p_0 \in L^1(\partial\Omega)} \bigg\{ \langle p_0, p_0^* \rangle - \int_{\Gamma} \phi(x, \nu^\Omega) \, |p_0 - f| \, d\mathcal{H}^{N-1} \bigg\} = $$
$$ = \sup_{p_0 \in L^1(\partial\Omega)} \bigg\{ \langle p_0 + f, p_0^* \rangle - \int_{\Gamma} \phi(x,\nu^\Omega) \, |p_0| \, d\mathcal{H}^{N-1} \bigg\} =$$
$$ = \langle f, p_0^* \rangle + \sup_{p_0 \in L^1(\partial\Omega)} \bigg\{ \langle p_0, p_0^* \rangle - \int_{\Gamma} \phi(x,\nu^\Omega) \, |p_0| \, d\mathcal{H}^{N-1} \bigg\} = $$
\begin{equation*}
= \twopartdef{\int_{\partial\Omega} f \, p_0^* \, d\mathcal{H}^{N-1}}{p_0^* = 0 \mbox{ a.e. on } \partial\Omega\backslash\Gamma, \, |p_0^*| \leq \phi(x,\nu^\Omega) \mbox{ a.e. on } \Gamma;}{+\infty}{\mbox{otherwise,}}    
\end{equation*}
which is the desired formula for $E_0^*$.
\end{proof}

In order to find the form of the dual problem \eqref{eq:dual}, the last thing we need to do is take a closer look at the operator $A^*$. This operator only enters the dual problem via $G^*(A^*p^*)$; by the form of $G^*$, we only need to check what is the condition so that $A^* p^* = 0$. By definition of the dual operator, for every $u \in W^{1,1}(\Omega)$ we have
$$ 0 = \langle u, A^* p^* \rangle  = \langle p^*, Au \rangle = \int_{\partial\Omega} p^*_0 \, u \, d\mathcal{H}^{N-1} + \int_\Omega \overline{p}^* \cdot \nabla u.$$
First, take $u$ to be a smooth function with compact support in $\Omega$; then, this condition reduces to
$$ \int_\Omega \overline{p}^* \cdot \nabla u = 0 \quad \forall u \in C_c^\infty(\Omega),$$
hence $\mathrm{div}(\overline{p}^*) = 0$ as distributions. Hence, $\overline{p}^* \in X_p(\Omega)$ for all $p \geq 1$ and we may use the Green's formula. Given any $u \in W^{1,1}(\Omega)$, we get 
$$ 0 = \langle u, A^* p^* \rangle  = \langle p^*, Au \rangle = \int_{\partial\Omega} p^*_0 \, u \, d\mathcal{H}^{N-1} + \int_\Omega \overline{p}^* \cdot \nabla u = \qquad\qquad\qquad\qquad\qquad\qquad\qquad\qquad $$
$$ \qquad\qquad\qquad\qquad\qquad = \int_{\partial\Omega} (p^*_0 + [\overline{p}^*, \nu^\Omega]) \, u \, d\mathcal{H}^{N-1} - \int_\Omega u \, \mathrm{div}(\overline{p}^*) = \int_{\partial\Omega} (p^*_0 + [\overline{p}^*, \nu^\Omega]) \, u \, d\mathcal{H}^{N-1}.$$
Because the right hand side disappears for all $u \in W^{1,1}(\Omega)$ and the trace operator from $W^{1,1}(\Omega)$ to $L^1(\partial\Omega)$ is surjective, we have $p_0^* = -[\overline{p}^*, \nu^\Omega]$.

We are now ready to state the form of the dual problem. Keeping in mind the above calculations, we first rewrite the dual problem \eqref{eq:dual} as
\begin{equation}
\sup_{p^* \in L^\infty(\partial\Omega) \times (L^\infty(\Omega))^N} \bigg\{ - E_0^*(-p_0^*) - E_1^*(\overline{p}^*) - G^*(A^* p^*) \bigg\}.
\end{equation}
Now, we take into account that unless $A^* p^* = 0$, this expression equals $-\infty$. Hence, we again rewrite the dual problem as
\begin{equation}
\sup_{\substack{ p^* \in L^\infty(\partial\Omega) \times (L^\infty(\Omega))^N \\ \mathrm{div}(\overline{p}^*) = 0 \\ p_0^* = -[\overline{p}^*,\nu^\Omega]}} \bigg\{ - E_0^*(-p_0^*) - E_1^*(\overline{p}^*) \bigg\}.
\end{equation}
Now, we take into account that unless $\phi^0(x,\overline{p}^*(x)) \leq 1$ a.e. in $\Omega$, then this expression equals $-\infty$. Hence, we again rewrite the dual problem as
\begin{equation}
\sup_{\substack{ p^* \in L^\infty(\partial\Omega) \times (L^\infty(\Omega))^N \\ \mathrm{div}(\overline{p}^*) = 0 \\ \phi^0(x,\overline{p}^*(x)) \leq 1 \mbox{ in } \Omega \\ p_0^* = -[\overline{p}^*,\nu^\Omega]}} \bigg\{ - E_0^*(-p_0^*) \bigg\}.
\end{equation}
Finally, we use Lemma \ref{lem:E0dual} and plug in the form of $E_0^*$ to the above formula. Let us denote the space of admissible vector fields in the dual problem as
\begin{equation*}
\mathcal{A}_\Gamma = \bigg\{ p^* \in L^\infty(\partial\Omega) \times (L^\infty(\Omega))^N: \, \mathrm{div}(\overline{p}^*) = 0 \mbox{ in } \mathcal{D}'(\Omega); \, \phi^0(x,\overline{p}^*(x)) \leq 1 \mbox{ a.e. in } \Omega; \qquad\qquad\qquad
\end{equation*}
\begin{equation*}
\qquad\qquad\qquad\qquad p_0^* = -[\overline{p}^*,\nu^\Omega]; \, p_0^* = 0 \,\, \mathcal{H}^{N-1}-\mbox{a.e. on } \partial\Omega\backslash\Gamma; \,  |p_0^*| \leq \phi(x,\nu^\Omega) \,\, \mathcal{H}^{N-1}-\mbox{a.e. on } \Gamma  \bigg\}
\end{equation*}
and obtain that the dual problem \eqref{eq:dual} takes the form
\begin{equation}\label{eq:finaldual}
\sup_{\substack{ p^* \in \mathcal{A}_\Gamma}} \bigg\{ \int_{\partial\Omega} f \, p_0^* \, d\mathcal{H}^{N-1} \bigg\}.
\end{equation}
Moreover, we have $\inf \eqref{eq:primal} = \sup \eqref{eq:dual}$ and the dual problem \eqref{eq:finaldual} admits a solution, because for $u_0 \equiv 0$ we have $E(Au_0) = \int_\Gamma \, \phi(x,\nu^\Omega) \, |f| \, d\mathcal{H}^{N-1} < \infty$, $G(u_0) = 0 < \infty$ and $E$ is continuous at $0$.

Furthermore, we may simplify the dual problem a bit. In light of the constraint $p_0^* = - [\overline{p}^*,\nu^\Omega]$ we may decrease the number of variables. Again, we introduce a similar set
\begin{equation*}
\mathcal{A}'_\Gamma = \bigg\{ \overline{p}^* \in (L^\infty(\Omega))^N: \, \mathrm{div}(\overline{p}^*) = 0 \mbox{ in } \mathcal{D}'(\Omega); \, \phi^0(x,\overline{p}^*(x)) \leq 1 \mbox{ a.e. in } \Omega; \qquad\qquad\qquad\qquad
\end{equation*}
\begin{equation*}
\qquad\qquad\qquad\qquad\qquad\qquad\qquad\qquad\qquad\qquad\qquad\qquad\qquad\qquad [\overline{p}^*,\nu^\Omega] = 0 \,\, \mathcal{H}^{N-1}-\mbox{a.e. on } \partial\Omega\backslash\Gamma  \bigg\}
\end{equation*}
The reduced form of the dual problem is
\begin{equation}\tag{RP*}\label{eq:reduceddual}
\sup_{\overline{p}^* \in \mathcal{A}'_\Gamma} \bigg\{ - \int_{\partial\Omega} f \, [\overline{p}^*, \nu^\Omega] \, d\mathcal{H}^{N-1} \bigg\}.
\end{equation}
After this identification, the last constraint in \eqref{eq:finaldual} disappears in \eqref{eq:reduceddual}, because it is implied by the constraint $\phi^0(x,\overline{p}^*(x)) \leq 1$ a.e. in $\Omega$.

\begin{remark}
Provided that $\mathcal{J}_\Gamma$ is lower semicontinuous and a solution of the primal problem exists $($in $W^{1,1}(\Omega))$, the extremality relations between any solution $u$ of the primal problem and any solution $p^*$ of the dual problem are as follows $($see \cite[Remark III.4.2]{ET}$)$:
\begin{equation}\label{eq:extremality1}
E(Au) + E^*(-p^*) = \langle -p^*, Au \rangle
\end{equation}
\begin{equation}\label{eq:extremality2}
G(u) + G^*(A^* p^*) = \langle u, A^* p^* \rangle.
\end{equation}
Let us plug in the expressions for $E, E^*, G, G^*$ into these two equations to see what are the extremality relations for problem \eqref{eq:variationalsolution}. Equation \eqref{eq:extremality2} is automatically satisfied and equation \eqref{eq:extremality1} becomes
\begin{equation*}
\int_\Gamma \phi(x,\nu^\Omega) \, |u - f| \, d\mathcal{H}^{N-1} + \int_\Omega \phi(x,\nabla u) \, dx - \int_{\partial\Omega} f \, p_0^* \, d\mathcal{H}^{N-1} = - \int_{\partial\Omega} u \, p_0^* \, d\mathcal{H}^{N-1} - \int_\Omega \overline{p}^* \cdot \nabla u \, dx,
\end{equation*}
where $p^* \in \mathcal{A}_\Gamma$. Keeping in mind that $p_0^* = 0$ $\mathcal{H}^{N-1}-$a.e. on $\partial\Omega\backslash\Gamma$, we rewrite this as
\begin{equation}
\int_{\Gamma} \bigg( \phi(x,\nu^\Omega) |u - f| + p_0^* (u - f)  \bigg) \, d\mathcal{H}^{N-1} + \int_\Omega \bigg( \phi(x,\nabla u) + \overline{p}^* \cdot \nabla u \bigg) \, dx = 0.
\end{equation}
Since $|p_0^*| \leq \phi(x,\nu^\Omega)$ $\mathcal{H}^{N-1}-$a.e. on $\Gamma$, the expression in the first bracket is nonnegative; similarly, because $\phi^0(x,\overline{p}^*) \leq 1$ a.e. in $\Omega$, the expression in the second bracket is nonnegative. Hence, both are equal to zero and we have
$$ \phi(x,\nabla u) = - \overline{p}^* \cdot \nabla u \quad \mbox{a.e. in } \Omega $$
and
$$ p_0^* \in \mbox{sign}(f - u) \phi(x,\nu^\Omega) \quad \mathcal{H}^{N-1}-\mbox{a.e. on } \Gamma.$$
Notice that because $p_0^* = - [\overline{p}^*, \nu^\Omega]$, the vector field $\mathbf{z} = -\overline{p}^*$ satisfies all the conditions in Definition \ref{dfn:1laplace}. Thus, if there exists a solution $u \in W^{1,1}(\Omega)$ of problem \eqref{eq:variationalsolution}, then it is a solution to the $1-$Laplace equation in the sense of Definition \ref{dfn:1laplace}.
\end{remark}

The following Theorem extends the observation above to the case when there is no minimiser of the functional $\mathcal{J}_\Gamma$ in $W^{1,1}(\Omega)$ and the minimum is in the space $BV(\Omega)$. Here, instead of the extremality conditions, we will use the $\varepsilon-$subdifferentiability property.

\begin{theorem}\label{thm:characterisation}
Suppose that $\Omega$ satisfies the Lipschitz extension property near $\Gamma$. For $u \in BV(\Omega)$, the following conditions are equivalent: \vspace{2mm} \\ 
(1) $u$ is a minimiser of the functional $\mathcal{J}_\Gamma$ $($in other words, $0 \in \partial\mathcal{J}_\Gamma(u))$; \\
(2) $u$ is a solution to the $1-$Laplace equation in the sense of Definition \ref{dfn:1laplace}. 
\end{theorem}

We need the Lipschitz extension property near $\Gamma$ so that the functional $\mathcal{J}_\Gamma$ is lower semicontinuous; it does not enter the proof directly. Moreover, this Theorem implies that under the Lipschitz extension property the $1-$Laplace equation admits a solution in the sense of Definition \ref{dfn:1laplace}.

\begin{proof}
$(1) \Rightarrow (2)$. Suppose that $p^*$ is a solution of the dual problem. We will prove that the vector field $\mathbf{z} = -\overline{p}^*$ satisfies all the conditions in Definition \ref{dfn:1laplace}. We see immediately that it satisfies the divergence constraint and that $[-\overline{p}^*,\nu^\Omega] = 0$ $\mathcal{H}^{N-1}$-a.e. on $\partial\Omega\backslash\Gamma$. Now, since $\mathcal{J}_\Gamma$ is lower semicontinuous, we may use the $\varepsilon-$subdifferentiability property of minimising sequences, see \cite[Proposition V.1.2]{ET}: for any minimising sequence $u_n$ for \eqref{eq:primal} and a maximiser $p^*$ of \eqref{eq:dual}, we have
\begin{equation}\label{eq:epsilonsubdiff1}
0 \leq E(Au_n) + E^*(-p^*) - \langle -p^*, Au_n \rangle \leq \varepsilon_n
\end{equation}
\begin{equation}\label{eq:epsilonsubdiff2}
0 \leq G(u_n) + G^*(A^* p^*) - \langle u_n, A^* p^* \rangle \leq \varepsilon_n
\end{equation}
with $\varepsilon_n \rightarrow 0$. Now, let $u \in BV(\Omega)$ be a minimiser of $\mathcal{J}_\Gamma$. Let us take the sequence $u_n \in W^{1,1}(\Omega)$ given by Lemma \ref{lem:anisotropicstrictapproximation}, i.e. it has the same trace as $u$ and converges $\phi-$strictly to $u$; then, it is a minimising sequence in \eqref{eq:primal}. Equation \eqref{eq:epsilonsubdiff2} is automatically satisfied and equation \eqref{eq:epsilonsubdiff1} gives
\begin{equation}
0 \leq \int_{\Gamma} \bigg( \phi(x,\nu^\Omega) |u_n - f| + p_0^* (u_n - f)  \bigg) \, d\mathcal{H}^{N-1} + \int_\Omega \bigg( \phi(x,\nabla u_n) + \overline{p}^* \cdot \nabla u_n \bigg) \, dx \leq \varepsilon_n.
\end{equation}
Because the trace of $u_n$ is fixed (and equal to the trace of $u$), the integral on $\Gamma$ does not change with $n$; hence, it has to equal zero. Keeping in mind that $p_0^* = [- \overline{p}^*, \nu^\Omega]$, we get
$$ [-\overline{p}^*, \nu^\Omega] \in \mbox{sign}(f - u) \phi(x,\nu^\Omega) \quad \mathcal{H}^{N-1}-\mbox{a.e. on } \Gamma.$$
The integral on $\Omega$ changes with $n$; since the first integral is zero, we have
\begin{equation}\label{eq:epsilonsubdiff3}
0 \leq \int_\Omega \bigg( \phi(x,\nabla u_n) + \overline{p}^* \cdot \nabla u_n \bigg) \, dx \leq \varepsilon_n.
\end{equation}
Finally, keeping in mind that $\mathrm{div}(\overline{p}^*) = 0$ and again using the fact that the trace of $u_n$ is fixed and equal to the trace of $u$, by Green's formula we get
$$ \int_\Omega \overline{p}^* \cdot \nabla u_n \, dx = \int_\Omega (\overline{p}^*, \nabla u_n \, dx) = \int_{\partial\Omega} [\overline{p}^*, \nu^\Omega] \, u_n \, d\mathcal{H}^{N-1} = \int_{\partial\Omega} [\overline{p}^*, \nu^\Omega] \, u \, d\mathcal{H}^{N-1} = \int_\Omega (\overline{p}^*, Du). $$
Hence, equation \eqref{eq:epsilonsubdiff3} takes the form
\begin{equation*}
0 \leq \int_\Omega \phi(x,\nabla u_n) \, dx - \int_\Omega (-\overline{p}^*, Du)   \leq \varepsilon_n
\end{equation*}
and since $u_n$ converges $\phi-$strictly to $u$, we get that
$$ \int_\Omega |Du|_\phi - \int_\Omega (-p^*, Du) = \lim_{n \rightarrow \infty} \bigg( \int_\Omega \phi(x,\nabla u_n) \, dx - \int_\Omega (-\overline{p}^*, Du) \bigg) = 0.$$
This together with Proposition \ref{prop:boundonanzelottipairing} implies that 
\begin{equation*}
(-\overline{p}^*, Du) = |Du|_\phi \quad \mbox{as measures in } \Omega,
\end{equation*}
so the pair $(u, -\overline{p}^*)$ satisfies all the conditions in Definition \ref{dfn:1laplace}. \vspace{2mm}

$(2) \Rightarrow (1)$. Suppose that $\mathbf{z}$ is the vector field from Definition \ref{dfn:1laplace}. Suppose that $w \in W^{1,1}(\Omega)$ and $w|_{\Gamma} = f$. By Green's formula
\begin{equation*}
0 = \int_{\Omega} (w-u) \mathrm{div}(\mathbf{z}) \, dx = -\int_\Omega (\mathbf{z}, Dw) + \int_\Omega (\mathbf{z}, Du) + \int_{\partial\Omega} [\mathbf{z},\nu^\Omega] (w-u) \, d\mathcal{H}^{N-1} = \qquad\qquad\qquad
\end{equation*}
\begin{equation*}
\qquad\qquad\qquad\qquad\qquad\qquad\qquad\qquad\qquad = -\int_\Omega (\mathbf{z}, Dw) + \int_\Omega (\mathbf{z}, Du) + \int_{\Gamma} [\mathbf{z},\nu^\Omega] (f-u) \, d\mathcal{H}^{N-1}.
\end{equation*}
We reorganise this equation to get
\begin{equation}\label{eq:laplaceimpliesminimum}
\mathcal{J}_\Gamma(u) = \int_\Omega (\mathbf{z}, Du) + \int_{\Gamma} [\mathbf{z},\nu^\Omega] (f-u) \, d\mathcal{H}^{N-1} = \int_\Omega (\mathbf{z}, Dw) \leq \int_\Omega |Dw|_\phi = \mathcal{J}_\Gamma(w),
\end{equation}
where the inequality follows from the fact that $\phi^0(x,\mathbf{z}(x)) \leq 1$ a.e. in $\Omega$. 

Now, suppose that $w \in BV(\Omega)$. By Proposition \ref{prop:approximation} there exists a sequence $w_n \in W^{1,1}(\Omega)$ with $w_n = f$ on $\Gamma$ such that $\mathcal{J}_\Gamma(w_n) \rightarrow \mathcal{J}_\Gamma(w)$. By equation \eqref{eq:laplaceimpliesminimum}, we have
\begin{equation*}
\mathcal{J}_\Gamma(u) \leq \mathcal{J}_\Gamma(w_n) \rightarrow \mathcal{J}_\Gamma(w),
\end{equation*}
so $u$ is a minimiser of the functional $\mathcal{J}_\Gamma$.
\end{proof}


In particular, because in the proof of Theorem \ref{thm:characterisation} we can use any solution of the dual problem, the structure of all solutions to \eqref{eq:variationalsolution} is determined by a (not uniquely chosen) single vector field. In the context of the standard least gradient problem this result has been first observed in \cite{MRL}. A proof using a duality-based approach can be found in \cite{Mor}; note that in order to deal with the case $\Gamma \neq \partial\Omega$ we need to use a different dualisation from the one used in \cite{Mor}. However, in the standard least gradient problem both methods lead to similar results and the dual problem in \cite{Mor} corresponds to the reduced dual problem \eqref{eq:reduceddual} in this paper. This structure result is formalised in the following Corollary; we will come back to the discussion about structure of solutions in Examples \ref{ex:square} and \ref{ex:structure}.

\begin{corollary}\label{cor:singlevectorfield}
The structure of all solutions to \eqref{eq:variationalsolution} is determined by a single vector field: if $p^*$ is a solution to the dual problem, then $-\overline{p}^*$ satisfies Definition \ref{dfn:1laplace} for any minimiser $u$ of the functional $\mathcal{J}_\Gamma$.
\end{corollary}

Now, let us see that solutions in the sense of Definition \eqref{dfn:1laplace} are functions of $\phi$-least gradient; this is formalised in the next Proposition. Denote by $u_\Gamma$ the restriction of the trace of $u$ to $\Gamma$.

\begin{proposition}
If $u$ satisfies Definition \ref{dfn:1laplace} and $u_{\Gamma} = f$, then it is a function of $\phi$-least gradient.
\end{proposition}

\begin{proof}
Let $v$ be a competitor in the definition of a function of $\phi$-least gradient. Then, because $u_\Gamma = v_\Gamma = f$, we have
$$ \int_\Omega |Du|_\phi = \mathcal{J}_\Gamma(u) \leq \mathcal{J}_\Gamma(v) = \int_\Omega |Dv|_\phi,$$
so $u$ is a function of $\phi$-least gradient.
\end{proof}

When $\Gamma = \partial\Omega$, the converse implication is also true, see \cite[Corollary 3.9]{Maz}, but it may fail when the boundary condition is imposed only on a subset of the boundary. To be more precise, we first state the following Proposition.

\begin{proposition}\label{prop:extendinggamma}
Suppose that $\Gamma \subset \Gamma' \subset \partial\Omega$. Let $f \in L^1(\Gamma)$. Suppose that $u \in BV(\Omega)$ is a solution to problem \eqref{eq:variationalsolution} for boundary data $f$ defined on $\Gamma$. Then, $u$ is a solution to problem \eqref{eq:variationalsolution} for boundary data $\widetilde{f} \in L^1(\Gamma')$ defined by the formula
$$ \widetilde{f} = \twopartdef{f}{\mbox{ on } \Gamma}{u}{\mbox{ on } \Gamma' \backslash \Gamma.}$$
\end{proposition}

\begin{proof}
Take $\mathbf{z} \in X_N(\Omega)$ relative to $u$ given by Definition \ref{dfn:1laplace}. We immediately see that $\mathbf{z}$ satisfies the conditions in Definition \ref{dfn:1laplace} for boundary data $\widetilde{f} \in L^1(\Gamma')$.
\end{proof}

In particular, the above Proposition implies that any solution to problem \eqref{eq:variationalsolution} is a solution to the least gradient problem defined on $\Gamma = \partial\Omega$ for some boundary data. However, in the other direction the implication is not true; if we restrict the boundary data to a subset of the boundary, solutions to problem \eqref{eq:variationalsolution} for $\partial\Omega$ may fail to be solutions for $\Gamma$. 

\begin{example}
Let $\Omega = B(0,1) \subset \mathbb{R}^2$. Let $\phi(x,\xi) = |\xi|$. Take
$$ \Gamma = \{ (x,y) \in \partial\Omega: y < 0 \}$$
and let $f \in C(\Gamma)$ be given by $f \equiv 0$. Let $u \in BV(\Omega)$ be given by the formula
$$ u(x,y) = \twopartdef{0}{\mbox{if } y < 0}{1}{\mbox{if } y > 0.}$$
Then $u$ is a function of least gradient, it satisfies $u_{\Gamma} = f$, but it is not a solution to Problem \ref{dfn:1laplace}; the only solution to Problem \ref{dfn:1laplace} is a function constant and equal to zero.
\end{example}

\section{Enforcing the trace condition}\label{sec:strongformulation}

In this Section, we are interested in the relationship between solutions to problem \eqref{problem} and problem \eqref{eq:variationalsolution}. To be more precise, we will prove that under some geometric assumptions on $\Omega$ and regularity assumptions on $f$, solutions to problem \eqref{eq:variationalsolution} attain the boundary condition in the sense of traces, so are in fact solutions to problem \eqref{problem}; in particular, solutions to problem \eqref{problem} exist. Let us stress that this does not automatically happen in the standard least gradient problem (when $\Gamma = \partial\Omega$); solutions in the Euler-Lagrange sense always exist, see \cite{MRL}, but solutions in the sense of traces may not exist if the boundary datum is not sufficiently regular, see \cite{ST}. In the second part of this Section, we study uniqueness and continuity properties of solutions to problem \eqref{problem}. These properties heavily depend on the shape of $\Gamma$ and we provide multiple examples in which the uniqueness and regularity properties known in the standard least gradient problem may fail. In particular, in order to have any positive results we need to assume that $\Gamma$ is connected.

\subsection{Existence in the trace sense}

A key geometric assumption in the anisotropic least gradient problem, which ensures that solutions are attained in the sense of traces for continuous boundary data, is the barrier condition introduced in \cite[Definition 3.1]{JMN}. In order to proceed, we will need a few versions of this condition; in the notation introduced below, the condition in \cite{JMN} the global barrier condition.

\begin{definition}\label{def:barrier}
Let $\Omega \subset \mathbb{R}^N$ be an open bounded set with Lipschitz boundary and let $x_0 \in \partial\Omega$. \\
(1) We say that $\Omega$ satisfies the pointwise barrier condition at $x_0$, if for sufficiently small $\varepsilon > 0$, if $V$ minimises $P_\phi(\cdotp; \mathbb{R}^N)$ in 
\begin{equation}\label{eq:barrier}
\{ W \subset \Omega: W \backslash B(x_0, \varepsilon) = \Omega \backslash B(x_0, \varepsilon) \},    
\end{equation}
then
$$\partial V^{(1)} \cap \partial\Omega \cap B(x_0, \varepsilon) = \emptyset.$$
(2) We say that $\Omega$ satisfies the barrier condition near $\Gamma \subset \partial\Omega$, if it satisfies the pointwise barrier condition at every $x_0 \in \Gamma$. \\
(3) We say that $\Omega$ satisfies the global barrier condition, if it satisfies the pointwise barrier condition at every $x_0 \in \partial\Omega$.
\end{definition}

We stress that the barrier condition depends on the choice of $\phi$. Keeping in mind that this definition is local in nature and $\Gamma \subset \partial\Omega$ is relatively open, we will modify the proof in \cite{JMN} to obtain existence of solutions to problem \eqref{problem} for continuous boundary data.

Now, let us give an illustration of the various versions of the barrier condition in the isotropic case, i.e. when $\phi(x,\xi) = |\xi|$ (an anisotropic discussion on the global barrier condition can be found in \cites{Gor2019IUMJ,Gor2020NA,JMN}). Assume that $\Omega \subset \mathbb{R}^2$; in two dimensions, the only connected minimal surfaces are line segments, so the boundaries of (isotropic) area-minimising sets are unions of line segments and the geometrical situation is easier to read. Then, the global barrier condition is equivalent to strict convexity of $\Omega$. When $\Gamma \subset \partial\Omega$ is a finite union of disjoint arcs, the barrier condition near $\Gamma$ is equivalent to positive mean curvature on a dense subset of $\Gamma$. Note that outside $\Gamma$ we do not assume anything about the shape of the domain, in particular the domain in Figure \ref{fig:barrierconditionneargamma} (with $\Gamma = \Gamma_1 \cup \Gamma_2$) satisfies the barrier condition near $\Gamma$, but does not satisfy the barrier condition near $\partial\Omega$ (i.e. the version introduced in \cite{JMN}). Finally, notice that the pointwise barrier condition is satisfied when $\Omega \cap B(x_0,r)$ is strictly convex for sufficiently small $r$. The converse statement does not hold, as we can see by taking $\Omega$ to be a convex polygon; then, the pointwise barrier condition is satisfied exactly at its corners. 

\begin{figure}
    \centering
    \includegraphics[scale=0.21]{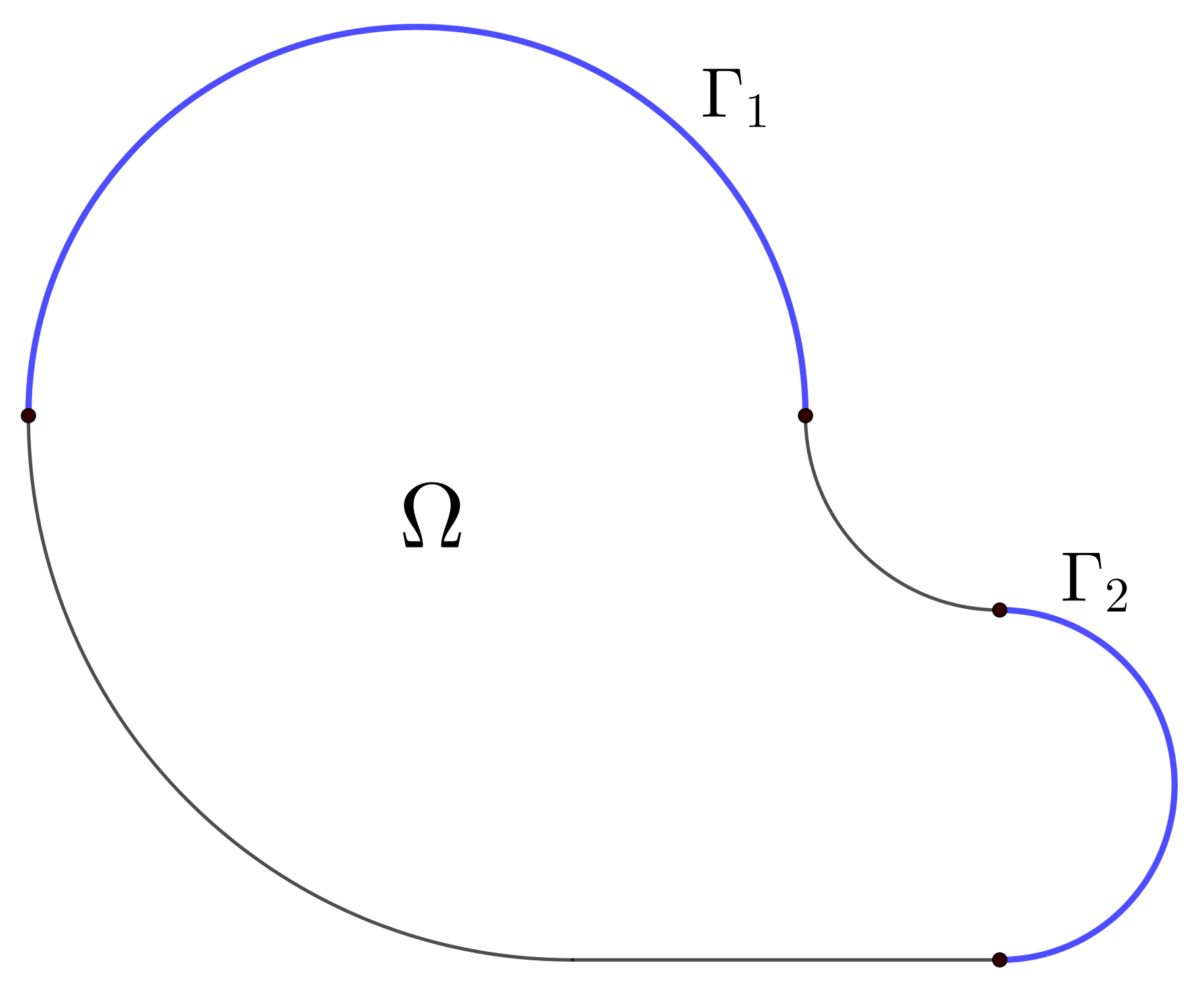} 
    \caption{Barrier condition near $\Gamma$}
    \label{fig:barrierconditionneargamma}
\end{figure}

Now, we prove two Lemmas that we will use in the proof of Theorem \ref{thm:existencetracesense}, which is the main result in this Section. The first Lemma is a corollary of Gagliardo's extension theorem, which gives us an extension with special properties near a fixed point $x_0 \in \partial\Omega$.

\begin{lemma}\label{lem:extension}
Suppose that $\Omega \subset \mathbb{R}^N$ is an open bounded set with Lipschitz boundary. Suppose that $g \in L^1(\partial\Omega)$. Suppose that $g$ is continuous at $x_0 \in \partial\Omega$. Then, there exists a function $\psi \in BV(\mathbb{R}^N \backslash \overline{\Omega})$ such that $\psi = g$ on $\partial\Omega$ and $\psi$ is continuous at $x_0$ in the following sense:
\begin{equation*}
\lim_{r \rightarrow 0} \mathrm{ess \, sup}_{y \in B(x_0,r) \backslash \overline{\Omega}} \, |\psi(y) - g(x_0)| = 0.    
\end{equation*}
\end{lemma}

\begin{proof}
Since $\partial\Omega$ is Lipschitz, take any extension $\psi \in W^{1,1}(\mathbb{R}^N \backslash \overline{\Omega})$ given by the Gagliardo extension theorem. In particular, $\psi = g$ on $\partial\Omega$. Now, assume that $g$ is continuous at $x_0 \in \partial\Omega$. Fix a sequence of positive numbers $\varepsilon_n$ such that $\varepsilon_n \rightarrow 0$ as $n \rightarrow \infty$. For every $n$, there exists $\delta_n > 0$ such that for $x \in B(x,\delta_n) \cap \partial\Omega$ we have
\begin{equation*}
g(x_0) - \varepsilon_n \leq g(x) \leq g(x_0) + \varepsilon_n.
\end{equation*}
Then, we modify the function $\psi$ in the following way. Let $\psi_0 = \psi$. We set $\psi_n \in BV(\mathbb{R}^N \backslash \overline{\Omega})$ by the formula
\begin{equation*}
\psi_n(y) = \threepartdef{g(x_0) + \varepsilon_n}{\mbox{if } y \in B(x_0,\delta_n / 2) \mbox{ and } \psi(y) > g(x_0) + \varepsilon_n;}{g(x_0) - \varepsilon_n}{\mbox{if } y \in B(x_0,\delta_n / 2) \mbox{ and } \psi(y) < g(x_0) - \varepsilon_n;}{\psi_{n-1}(y)}{\mbox{otherwise.}}
\end{equation*}
Notice that $\psi_n = g$ on $\partial\Omega$. Moreover, $\psi_n$ converges in $BV(\mathbb{R}^N \backslash \overline{\Omega})$ to some $\widetilde{\psi} \in BV(\mathbb{R}^N \backslash \overline{\Omega})$; in particular, also $\widetilde{\psi} = g$ on $\partial\Omega$. Extracting a sequence which converges almost everywhere, by definition of $\psi_n$ it is immediate that $\widetilde{\psi}$ satisfies the statement of the Lemma.
\end{proof}

The second Lemma is a variant of \cite[Lemma 3.4]{JMN} taking into account the definition of the pointwise barrier condition at $x_0$. Here, $E^{(1)}$ is the set of points of density one of $E$. 

\begin{lemma}\label{lem:beforeexistence}
Let $\Omega \subset \mathbb{R}^N$ be an open bounded set with Lipschitz boundary which satisfies the pointwise barrier condition at $x_0 \in \partial\Omega$. Assume that $E \subset \mathbb{R}^N$ is $\phi$-area-minimising in $\Omega$. Then, if $x \in \partial E^{(1)}$, then for all $\varepsilon > 0$ we have
\begin{equation}
(B(x_0,\varepsilon) \cap \partial E^{(1)}) \backslash \overline{\Omega} \neq \emptyset.
\end{equation}
\end{lemma}

In other words, it is impossible that for some $\varepsilon > 0$ we have $B(x_0,\varepsilon) \cap E^{(1)} \subset \overline{\Omega}$. In the isotropic case in two dimensions, so that boundaries of area-minimising sets are unions of line segments, this can be rephrased as follows: two connected components of $\partial E$ cannot intersect at a boundary point. 

\begin{proof}
First, notice that the problem of minimisation of $P_\phi(\cdot, \mathbb{R}^N)$ in \eqref{eq:barrier} always has a solution. Assume that $V_n$ is a minimising sequence in problem \eqref{eq:barrier}; then, the union $V = \bigcup_{n=1}^\infty V_n$ is also admissible and the claim follows by lower semicontinuity of $\phi$-total variation.

Assume otherwise: let $x_0 \in \partial\Omega \cap \partial E^{(1)}$ be such that $B(x_0,\varepsilon) \cap \partial E^{(1)} \subset \overline{\Omega}$ for some $\varepsilon > 0$. Let $V$ minimise $P_\phi(\cdot, \mathbb{R}^N)$ in \eqref{eq:barrier}. Then, by \cite[Lemma 2.5]{JMN} the set $V' = V \cup (E \cap \Omega)$ also minimises $P_\phi(\cdot, \mathbb{R}^N)$ in \eqref{eq:barrier}. However, $x_0 \in \partial V'^{(1)}$, which contradicts the barrier condition near $\Gamma$.
\end{proof}

The Theorem below is the main result of this Section. Here, we give pointwise result concerning the way in which the trace of the solution is attained at the boundary. The result is stated in a fairly general setting, including existence of solutions to problem \eqref{problem} for continuous boundary data. This Theorem is a generalisation of \cite[Theorem 1.1]{JMN} in two ways: the first part of this result concerns local properties of the solution around a continuity point of the boundary datum and is new even in the context of the standard anisotropic least gradient problem, i.e. when $\Gamma = \partial\Omega$ (also note that in that case the assumptions are simpler, because the Lipschitz extension property is automatically satisfied). The second part provides an existence result for the anisotropic least gradient problem in the case when $\Gamma \neq \partial\Omega$.

\begin{theorem}\label{thm:existencetracesense}
Let $\Omega \subset \mathbb{R}^N$ be an open bounded set with Lipschitz boundary. Let $\Gamma \subset \partial\Omega$ be relatively open and suppose that $\Omega$ satisfies the Lipschitz extension property near $\Gamma$. Then: \\
(1) Suppose that $\Omega$ satisfies the pointwise barrier condition at $x_0$. If $f \in L^1(\Gamma)$ is continuous at $x_0$ and $u \in BV(\Omega)$ is a solution to problem \eqref{eq:variationalsolution}, then $u(x_0) = f(x_0)$. Moreover, $u$ is continuous at $x_0$ in the sense that
\begin{equation*}
\lim_{r \rightarrow 0} \mathrm{ess \, sup}_{y \in B(x_0,r) \cap \Omega} \, |u(y) - f(x_0)| = 0.    
\end{equation*}
(2) Suppose that $\Omega$ satisfies the barrier condition near $\Gamma$. If $f \in L^1(\Gamma)$ is continuous $\mathcal{H}^{N-1}$-a.e. on $\Gamma$, then every solution to problem \eqref{eq:variationalsolution} is a solution to problem \eqref{problem}. In particular, there exists a solution to problem \eqref{problem}.
\end{theorem}

\begin{proof}
(1) Extend $f \in L^1(\Gamma)$ to a function on $L^1(\partial\Omega)$ by setting it equal to zero on $\partial\Omega\backslash\Gamma$ (without changing the notation). Let $\psi \in BV(\mathbb{R}^N \backslash \overline{\Omega})$ be the extension of $f$ given by Lemma \ref{lem:extension}. Let $\Omega'$ be the set given in Definition \ref{dfn:lipschitzextension}. Since $u$ is a solution to problem \eqref{eq:variationalsolution}, it is also a solution to the problem
\begin{equation*}
\min \bigg\{ \int_{\Omega'} |Dw|_\phi: \, w = \psi \mbox{ on } \Omega' \backslash \overline{\Omega} \, \bigg\}.
\end{equation*}
By Theorem \ref{thm:anisobgg}, superlevel sets are $\phi$-area-minimising in $\Omega$. Now, suppose that $u(x_0) \neq f(x_0)$; then, there exists $\delta > 0$ such that
\begin{equation}\label{eq:nowweusebarrier}
\lim_{r \rightarrow 0} \mathrm{ess \, sup}_{y \in B(x_0,r) \backslash \overline{\Omega}} \, (u(y) - f(x_0)) \geq \delta \mbox{   or   } \lim_{r \rightarrow 0} \mathrm{ess \, sup}_{y \in B(x_0,r) \backslash \overline{\Omega}} \, (f(x_0) - u(y)) \leq -\delta.    
\end{equation}
Assume that the first condition holds (the second one is handled similarly). Let $E = E_{f(x) + \delta /2}$. Notice that $x \in \partial E^{(1)}$; by equation \eqref{eq:nowweusebarrier}, we either have $x \in E^(1)$ or $x \in \partial E^(1)$, but the former possibility is excluded by our choice of $\psi$, because it satisfies the statement of Lemma \ref{lem:extension}. Now, we apply Lemma \ref{lem:beforeexistence} to conclude that for all $r > 0$ we have $(B(x_0,r) \cap \partial E^{(1)}) \backslash \overline{\Omega} \neq \emptyset$. We reach a contradiction, because Lemma \ref{lem:extension} implies that for sufficiently small balls centered at $x_0$ the extension $\psi$ only admits values smaller than $f(x) + \delta /4$ except for a set of level zero. \\
(2) This immediately follows from the first part: because $\Omega$ satisfies the Lipschitz extension property near $\Gamma$, problem \eqref{eq:variationalsolution} admits a solution. Let $u \in BV(\Omega)$ be a solution to problem \eqref{eq:variationalsolution}. Then, by point (1), we have $u = f$ $\mathcal{H}^{N-1}$-a.e. on $\Gamma$, so $u$ is a solution to problem \eqref{problem}. 
\end{proof}

\subsection{Regularity and structure of solutions}

Now, we turn to the issue of regularity and structure of solutions. Here, the situation is more complicated; even when $\Gamma = \partial\Omega$, then regularity of solutions depends on exact assumptions on $\phi$. Suppose that the boundary data are continuous. In the isotropic case, if the domain is strictly convex, solutions are continuous in $\overline{\Omega}$ in any dimension $N \geq 2$, see \cite{SWZ} (a similar result holds in the weighted case for sufficiently smooth weights, see \cite{Zun}). In the anisotropic case, if $\phi$ is uniformly convex and smooth enough, then solutions in dimensions two and three are continuous in $\overline{\Omega}$ (assuming that $\partial\Omega$ is connected), see \cite{JMN}. Moreover, under some additional assumptions H\"older regularity of boundary data implies H\"older regularity of solutions (with a worse exponent), see \cites{FM,Gor2020NA,SWZ}. 

When $\Gamma \neq \partial\Omega$, a few additional phenomena appear. To highlight them, we now present a series of examples. All the examples are isotropic, i.e. throughout this subsection we have $\phi(x,\xi) = |\xi|$; the loss of regularity of solutions compared to the standard least gradient problem is inherent to the case $\Gamma \neq \partial\Omega$ and is not caused by the lack of regularity or uniform convexity of the anisotropy.

The first example is very simple and it illustrates two issues. The first one is that in order to prove any regularity or uniqueness results for continuous boundary data, we need to assume that $\Gamma$ is connected.
The second is that even though we have a structure result (Corollary \ref{cor:singlevectorfield}) of the same type as for the full least gradient problem, in the case when $\Gamma \neq \partial\Omega$ it gives us far less information.

\begin{example}\label{ex:square}
Let $\Omega = (0,1)^2 \subset \mathbb{R}^2$. Let $\phi(x,\xi) = |\xi|$. Suppose that $\Gamma = \Gamma_0 \cup \Gamma_1$, where
\begin{equation*}
\Gamma_0 = \{ (x,y) \in \partial\Omega: \, y = 0 \}, \qquad\qquad \Gamma_1 = \{ (x,y) \in \partial\Omega: \, y = 1 \}.
\end{equation*}
Suppose that $f \in C(\Gamma)$ is as follows: $f \equiv 0$ on $\Gamma_0$ and $f \equiv 1$ on $\Gamma_1$. Then, any function $u(x,y) = u(y)$ which is an increasing function of $y$ such that $u(x,0)=0$ and $u(x,1)=1$ satisfies Definition \ref{dfn:1laplace} with the vector field $\mathbf{z} = (0,1)$. In particular, there exist multiple solutions to the least gradient problem and they may fail to be continuous inside $\Omega$.
\end{example}

This example has the virtue of simplicity, but one may point out that even though we lose continuity of some solutions, there exist solutions which are continuous (for instance $u(x,y)=y$). Moreover, $\Omega$ fails to satisfy the barrier condition near $\Gamma$. 

The second example is a version of the Brother-Marcellini example; for different versions of this example, see \cites{Mar,MRL,SZ}. We show that when $\Gamma$ is not connected, then even if domain is strictly convex (so the barrier condition is satisfied near the whole $\partial\Omega$), we may still lose continuity inside the domain and uniqueness of solutions. Moreover, unlike the previous example, there is no continuous solution to problem \eqref{problem}.

\begin{example}
Let $\Omega = B(0,1) \subset \mathbb{R}^2$. Let $\phi(x,\xi) = |\xi|$. Suppose that $\Gamma = \Gamma_0 \cup \Gamma_1$, where
\begin{equation*}
\Gamma_0 = \{ (x,y) \in \partial\Omega: \, y < - \frac{1}{\sqrt{2}} \}, \qquad\qquad \Gamma_1 = \{ (x,y) \in \partial\Omega: \, y > \frac{1}{\sqrt{2}} \}.
\end{equation*}
Suppose that $f \in C(\Gamma)$ is as follows: $f \equiv 0$ on $\Gamma_0$ and $f \equiv 1$ on $\Gamma_1$. Then, there exist multiple solutions to the least gradient problem and they are not continuous inside $\Omega$. Namely, for $\lambda \in [0,1]$ let $u_\lambda \in BV(\Omega)$ be defined by the formula
\begin{equation*}
u_\lambda(x,y) = \threepartdef{0}{\mbox{if } y < \frac{1}{\sqrt{2}};}{\lambda}{\mbox{if } -\frac{1}{\sqrt{2}} < y < \frac{1}{\sqrt{2}};}{1}{\mbox{if } y > \frac{1}{\sqrt{2}}.}
\end{equation*}
Then, $u_\lambda$ is a solution to problem \eqref{eq:variationalsolution}. To see this, take
\begin{equation}\label{eq:brotherexample}
\mathbf{z}(x,y) = \twopartdef{(0,1)}{\mbox{if } |x| < \frac{1}{\sqrt{2}}}{(0,0)}{\mbox{if } |x| > \frac{1}{\sqrt{2}}.}    
\end{equation}
Then, we have $\mathrm{div}(\mathbf{z}) = 0$, $(\mathbf{z}, Du_\lambda) = |Du_\lambda|$, $[\mathbf{z},\nu^\Omega] = 0$ on $\partial\Omega \backslash \Gamma$ and $u = f$ on $\Gamma$, so $\mathbf{z}$ satisfies all the conditions in Definition \ref{dfn:1laplace}. In particular, the solution is not unique and all the solutions $u_\lambda$ are discontinuous inside $\Omega$.

Furthermore, every solution to problem \eqref{problem} is of the form $u_\lambda$ for some $\lambda \in [0,1]$. We can see this in the following way: by Corollary \ref{cor:singlevectorfield}, any solution $u \in BV(\Omega)$ to problem \eqref{eq:variationalsolution} satisfies Definition \ref{dfn:1laplace} with the vector field $\mathbf{z}$. Hence, it is constant in $\{ (x,y) \in \Omega: |x| > \frac{1}{\sqrt{2}} \}$ and it is a function of only one variable, $u(x,y) = u(y)$, in $\{ (x,y) \in \Omega: |x| < \frac{1}{\sqrt{2}} \}$; moreover, it is increasing as a function of $y$. In this class of functions, in order for $u$ to minimise the functional $\mathcal{J}_\Gamma$, it has to be constant on the square $\{ (x,y) \in \Omega: |x|,|y| < \frac{1}{\sqrt{2}} \}$ and we easily see that it is of the form $u_\lambda$ for some $\lambda \in [0,1]$.
In particular, all the solutions to problem \eqref{eq:variationalsolution} are discontinuous inside $\Omega$.
\end{example}

The Example above serves as an illustration to Proposition \ref{prop:extendinggamma}: the vector field $\mathbf{z}$ given by formula \eqref{eq:brotherexample} also works in the classical least gradient problem, when $\Gamma = \partial\Omega$, with boundary datum equal to the trace of $u_\lambda$ (it satisfies Definition \ref{dfn:1laplace} with $\Gamma = \partial\Omega$). Moreover, in that case it also satisfies Definition \ref{dfn:1laplace} with the vector field $\mathbf{z}' = (0,1)$, but $\mathbf{z}'$ does not satisfy Definition \ref{dfn:1laplace} with $\Gamma$ as in the above Example.

Since in the case when $\Gamma$ is not connected we may lose continuity of solutions inside $\Omega$, we now turn our attention to the case when $\Gamma$ is connected; in two dimensions, this means that $\Gamma$ is an arc. Some analysis in the isotropic case, using a different method based on a variant of the Sternberg-Williams-Ziemer construction (see \cite{SWZ}), has been done in \cite{GRS2017NA}; in particular, it was observed that even for simplest boundary data, we lose continuity of solutions in $\overline{\Omega}$; discontinuities naturally form at the ends of $\Gamma$, see \cite[Theorem 3.2]{GRS2017NA}. This phenomenon is presented in the following Example, with special emphasis on how Definition \ref{dfn:1laplace} works in this case. 

\begin{example}\label{ex:arc}
Let $\Omega = B(0,1) \subset \mathbb{R}^2$. Let $\phi(x,\xi) = |\xi|$. We set
\begin{equation*}
\Gamma = \{ (x,y) \in \partial\Omega: \, y < 0 \}.
\end{equation*}
Take boundary data $f \in C(\Gamma)$ equal to $f(x,y) = x$. The situation is as regular as possible: $\Gamma$ is an arc, the boundary datum is smooth and it admits a Lipschitz continuous extension to $\overline{\Gamma}$.

Nonetheless, the solution is not continuous in $\overline{\Omega}$. We define $u$ in the following way: for any $t \in (0,1)$, we denote by $l_t$ the line segment between the points $(1,0)$ and $f^{-1}(t) = (t,-\sqrt{1-t^2})$. Similarly, for any $t \in (-1,0)$, we denote by $l_t$ the line segment between the points $(-1,0)$ and $f^{-1}(t) = (t,-\sqrt{1-t^2})$. Then, for $t \in (-1,1) \backslash \{ 0 \}$ we set $u$ to be equal to $t$ on $l_t$. Finally, we set $u$ to equal $0$ on the rest of the domain. The function $u$ constructed in this way is smooth in $\Omega$ and has two discontinuity points in $\overline{\Omega}$: $(-1,0)$ and $(1,0)$, the endpoints of $\Gamma$. 

Let us see that $u \in C^\infty(\Omega) \cap W^{1,1}(\Omega)$ constructed in this way is a solution to problem \eqref{problem}. Since the boundary data are continuous, by Theorem \ref{thm:existencetracesense} it suffices to check that it satisfies Definition \ref{dfn:1laplace}. To this end, divide $\Omega$ into four parts (and a set of codimension one) as follows:
\begin{equation*}
\Omega_1 = \{ (x,y) \in \Omega: \, y < |x| - 1 \}, \qquad \Omega_2 = \{ (x,y) \in \Omega: \, x < 0, \, - x - 1 < y < x + 1 \},
\end{equation*}
\begin{equation*}
\Omega_3 = \{ (x,y) \in \Omega: \, x > 0, \, x - 1 < y < 1 - x \},  \qquad  \Omega_4 = \{ (x,y) \in \Omega: \, y > 1 - |x| \}.
\end{equation*}
The situation is presented in Figure \ref{fig:rhombus}. We define the vector field $\mathbf{z} \in L^\infty(\Omega; \mathbb{R}^2)$ by the formula
\begin{equation*}
\mathbf{z}(x,y) = \left\{ \begin{array}{lll}
			\frac{\nabla u}{|\nabla u|} & \mbox{ in } \Omega_1; \\
			(\frac{1}{\sqrt{2}},\frac{1}{\sqrt{2}}) & \mbox{ in } \Omega_2; \\
			(\frac{1}{\sqrt{2}},-\frac{1}{\sqrt{2}}) & \mbox{ in } \Omega_3; \\
			(0,0) & \mbox{ in } \Omega_4.
		\end{array} \right.
\end{equation*}
Then, we have $\mathrm{div}(\mathbf{z}) = 0$, $(\mathbf{z},Du) = |Du|$, $[\mathbf{z},\nu^\Omega] = 0$ on $\partial\Omega\backslash\Gamma$ and $u = f$ on $\Gamma$, so $\mathbf{z}$ satisfies all the conditions in Definition \ref{dfn:1laplace}. Hence, $u$ is a solution to problem \eqref{problem}, which is not continuous in $\overline{\Omega}$. In particular, classical estimates regarding H\"older exponent for the solution (see \cite{SWZ}) cannot be true when $\Gamma \neq \partial\Omega$. However, note that the solution is continuous in $\Omega \cup \Gamma$; we will come back to this issue in the next subsection.

\begin{figure}
    \centering
    \includegraphics[scale=0.15]{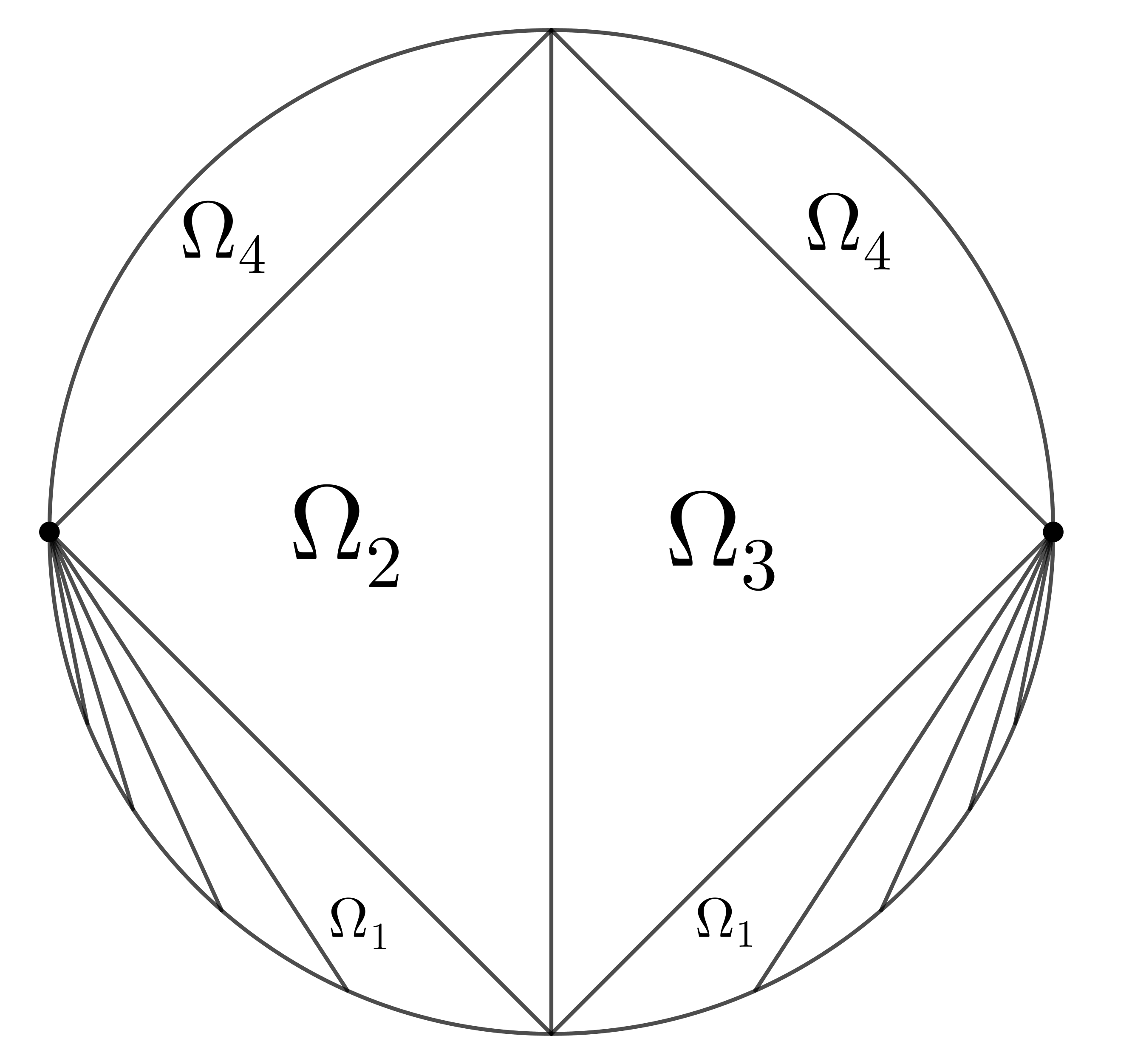} 
    \caption{Discontinuities at the ends of $\Gamma$}
    \label{fig:rhombus}
\end{figure}
\end{example}

In the previous Example, the discontinuity points were limited to points in $\partial\Gamma$. This need not be the case; depending on the shape of $\Omega$, discontinuities may also form in the interior of $\partial\Omega \backslash \Gamma$. This is presented in the following Example.

\begin{example}
Let $\Omega \subset \mathbb{R}^2$ be defined as follows:
\begin{equation*}
\Omega = \bigg\{ (x,y): \, y < 0, \, x^2 + y^2 < 1 \bigg\} \, \cup \, \bigg\{ (x,y): \, -1 < x < 0, \, 0 \leq y < \frac{1}{2\sqrt{3}} - \frac{1}{\sqrt{3}} \bigg|x + \frac{1}{2} \bigg| \bigg\} \, \cup \qquad\qquad
\end{equation*}
\begin{equation*}
\qquad\qquad\qquad\qquad\qquad\qquad\qquad\qquad\qquad\qquad \cup \, \bigg\{ (x,y): \, 0 < x < 1, \, 0 \leq y < \frac{1}{2\sqrt{3}} - \frac{1}{\sqrt{3}} \bigg|x - \frac{1}{2} \bigg| \bigg\}.
\end{equation*}
Let $\phi(x,\xi) = |\xi|$. We set $\Gamma = \{ (x,y) \in \partial\Omega: \, y < 0 \}$, so that $\Omega$ satisfies the barrier condition near $\Gamma$, and take boundary data $f \in C(\Gamma)$ equal to $f(x,y) = x$. The situation is similar to the previous Example, but the shape of the domain enforces a different structure of the solution.

We define $u$ in the following way: for any $t \in (\frac{1}{2},1)$, we denote by $l_t$ the line segment between the points $(1,0)$ and $f^{-1}(t) = (t,-\sqrt{1-t^2})$. Similarly, for any $t \in (-1,-\frac{1}{2})$, we denote by $l_t$ the line segment between the points $(-1,0)$ and $f^{-1}(t) = (t,-\sqrt{1-t^2})$. For $t \in (-\frac{1}{2},\frac{1}{2})$, we denote by $l_t$ the line segment between the points $(0,0)$ and $f^{-1}(t) = (t,-\sqrt{1-t^2})$. Then, for $t \in (-1,1) \backslash \{ -\frac{1}{2}, \frac{1}{2} \}$ we set $u$ to be equal to $t$ on $l_t$. Finally, we set $u$ to equal $\pm\frac{1}{2}$ on the remaining part of the domain, with positive sign when $x > 0$ and with negative sign when $x < 0$. The function $u$ constructed in this way is smooth in $\Omega$ and has three discontinuity points in $\overline{\Omega}$: $(-1,0)$, $(1,0)$ and $(0,0)$. 

Now, we check that $u \in C^\infty(\Omega) \cap W^{1,1}(\Omega)$ constructed in this way satisfies Definition \ref{dfn:1laplace}, so by Theorem \ref{thm:existencetracesense} it is a solution to problem \eqref{problem}. To this end, divide $\Omega$ into five parts (and a set of codimension one) as follows:
\begin{equation*}
\Omega_4 = \bigg\{ (x,y) \in \Omega: \, 0 < x < \frac{1}{2}, \, -\sqrt{3} x < y < \frac{1}{\sqrt{3}} x \bigg\}, \qquad \Omega_2 = \Omega_4 + (-1,0),
\end{equation*}
\begin{equation*}
\Omega_3 = \bigg\{ (x,y) \in \Omega: \, -\frac{1}{2} < x < 0, \, \sqrt{3} x < y < -\frac{1}{\sqrt{3}} x \bigg\}, \qquad \Omega_5 = \Omega_3 + (1,0).
\end{equation*}
Here, the notation $\Omega_i + (\pm 1,0)$ means simply that we shift the set $\Omega_i$ by the vector $(\pm 1,0)$. Finally, we set $\Omega_1 = \Omega \backslash (\Omega_2 \cup \Omega_3 \cup \Omega_4 \cup \Omega_5)$. The situation is presented in Figure \ref{fig:corner}. We define the vector field $\mathbf{z} \in L^\infty(\Omega; \mathbb{R}^2)$ by the formula
\begin{equation*}
\mathbf{z}(x,y) = \threepartdef{\frac{\nabla u}{|\nabla u|}}{\mbox{ in } \Omega_1;}{(\frac{\sqrt{3}}{2},\frac{1}{2})}{\mbox{ in } \Omega_2 \cup \Omega_4;}{(\frac{\sqrt{3}}{2},-\frac{1}{2} )}{\mbox{ in } \Omega_3 \cup \Omega_5.}
\end{equation*}
Such vector field $\mathbf{z}$ satisfies all the conditions in Definition \ref{dfn:1laplace}, so $u$ is a solution to problem \eqref{problem}. It is discontinuous not only at points in $\partial\Gamma$, but also in the interior of $\partial\Omega\backslash\Gamma$.

\begin{figure}
    \centering
    \includegraphics[scale=0.15]{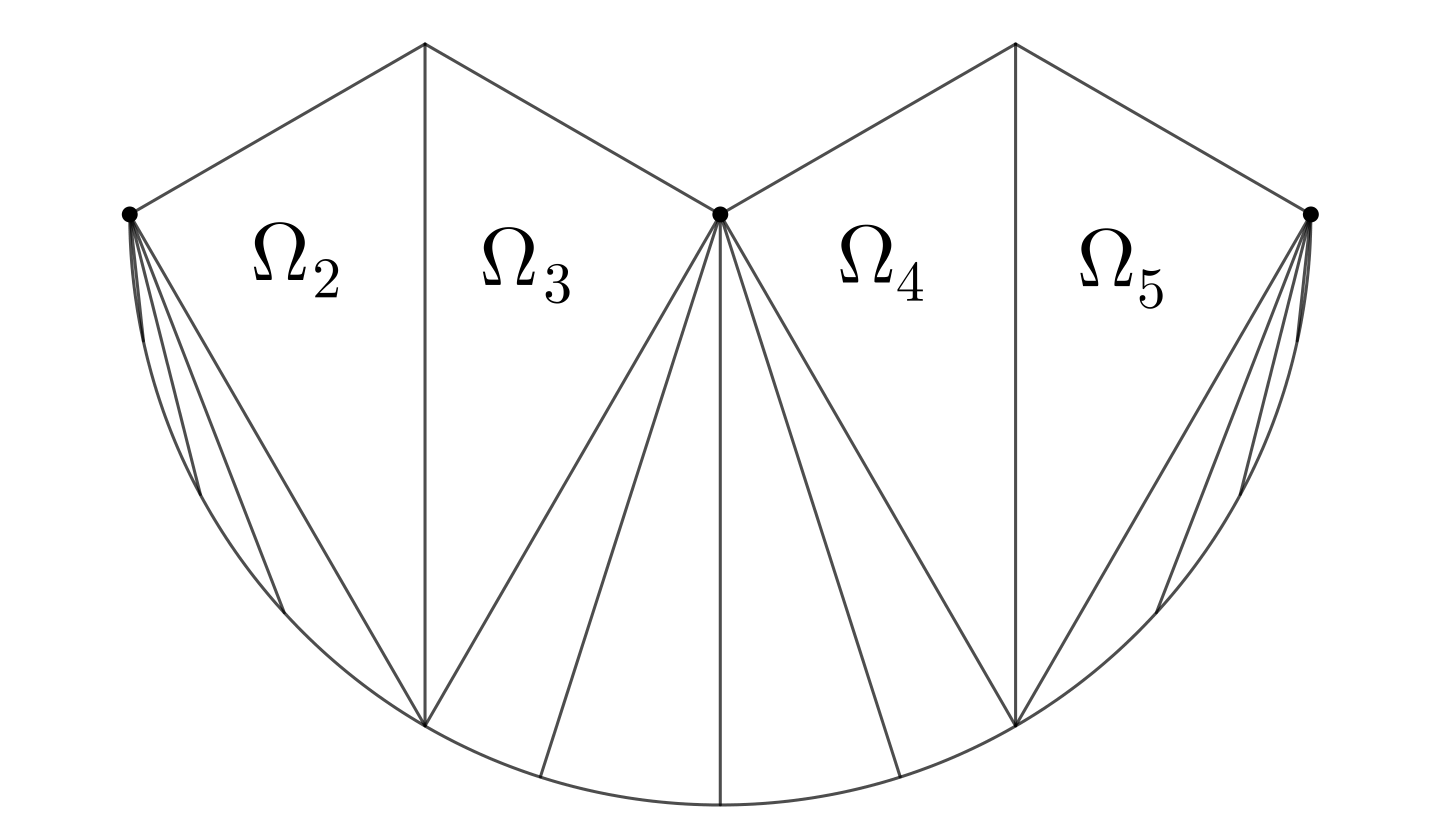} 
    \caption{Discontinuities may form in the interior of $\partial\Omega\backslash\Gamma$}
    \label{fig:corner}
\end{figure}
\end{example}

The final Example concerns the structure of solutions to problem \eqref{problem}. In the study of the least gradient problem, there is a well-known conjecture formulated by Maz\'on, Rossi and Segura de Le\'on in \cite{MRL} regarding the structure of solutions. It can be formulated as follows: even when the boundary datum is discontinuous, all the solutions share the same frame of superlevel sets. In the original version formulated by the authors, this was expressed by the fact that we can use the same vector field $\mathbf{z}$ in the $1$-Laplace formulation (as in Definition \ref{dfn:1laplace}); this has been since proved multiple times in different contexts, see \cites{FM,Mor} and in this paper in Corollary \ref{cor:singlevectorfield}. 

A stronger version of the conjecture, which gives some information about pointwise behaviour of solutions, was proved in \cite{Gor2018JMAA} for the isotropic least gradient problem in low dimensions. Namely, on convex domains two solutions of the least gradient problem agree except for a set on which they are both locally constant. However, even though the weaker version of the conjecture is true when $\Gamma \neq \partial\Omega$, the stronger version fails even in two dimensions. This is easy to see when $\Gamma$ is not connected, see Example \ref{ex:square}; however, the stronger version of the conjecture may fail even when the domain is strictly convex and $\Gamma$ is an arc. This is presented in the following Example.

\begin{example}\label{ex:structure}
Let $\Omega = \Omega_1 \cup \Omega_2 \cup \Omega_3 \subset \mathbb{R}^2$, where
$$ \Omega_1 = B((0,0),2) \cap \{ (x,y): \, x,y > 0 \}, $$
$$ \Omega_2 = B ((2,1),\sqrt{5}) \cap \{ (x,y): x \leq 0 \}, \qquad \Omega_3 = B((1,2),\sqrt{5}) \cap \{ (x,y): y \leq 0 \}. $$
Let $\phi(x,\xi) = |\xi|$. The domain is strictly convex, so it satisfies the barrier condition. Let
$$ \Gamma = \partial\Omega \cap (\{ (x,y): xy <0 \} \cup \{ (0,0) \}),$$
so $\Gamma$ is an arc, in particular it is connected. We define $f \in L^\infty(\Gamma)$ by the formula
\begin{equation*}
f(x,y) = \twopartdef{1}{\mbox{ if } x < 0;}{0}{\mbox{ if } x > 0.}
\end{equation*}
The boundary datum is continuous $\mathcal{H}^1$-a.e., so by Theorem \ref{thm:existencetracesense} there exists a solution to problem \eqref{problem}. 

Denote by $(r,\theta)$ the polar coordinates on $\mathbb{R}^2$. Assume that $u \in BV(\Omega)$ is of the form
\begin{equation*}
u(r,\theta) = \threepartdef{0}{\mbox{ if } \theta < 0;}{g(\theta)}{\mbox{ if } \theta \in [0,\frac{\pi}{2}];}{1}{\mbox{ if } \theta > \frac{\pi}{2},}
\end{equation*}
where $g: [0, \frac{\pi}{2}] \rightarrow [0,1]$ is an increasing function of $\theta$. Then, $u$ satisfies Definition \ref{dfn:1laplace} with the vector field
\begin{equation*}
\mathbf{z}(r,\theta) = \threepartdef{(0,1)}{\mbox{ if } \theta < 0;}{\hat{e}_\theta}{\mbox{ if } \theta \in [0,\frac{\pi}{2}];}{(-1,0)}{\mbox{ if } \theta > \frac{\pi}{2},}
\end{equation*}
where $\hat{e}_\theta$ is the unit vector in the direction $\theta$. In particular, there are uncountably many solutions to problem \eqref{problem} with different structure of level sets. The situation is presented in Figure \ref{fig:niejmaa}; all the boundaries of superlevel sets are line segments from $(0,0)$ to points in $\partial\Omega \backslash \Gamma$. 
\end{example}

We can modify the above Example in such a way that $\partial\Omega$ is smooth; failure of the structure theorem \cite[Theorem 1.1]{Gor2018JMAA} is caused only by the fact that the distance from the discontinuity point $(0,0)$ to any point in the Neumann part of the boundary $\partial\Omega\backslash\Gamma$ is the same.

\begin{figure}
    \centering
    \includegraphics[scale=0.25]{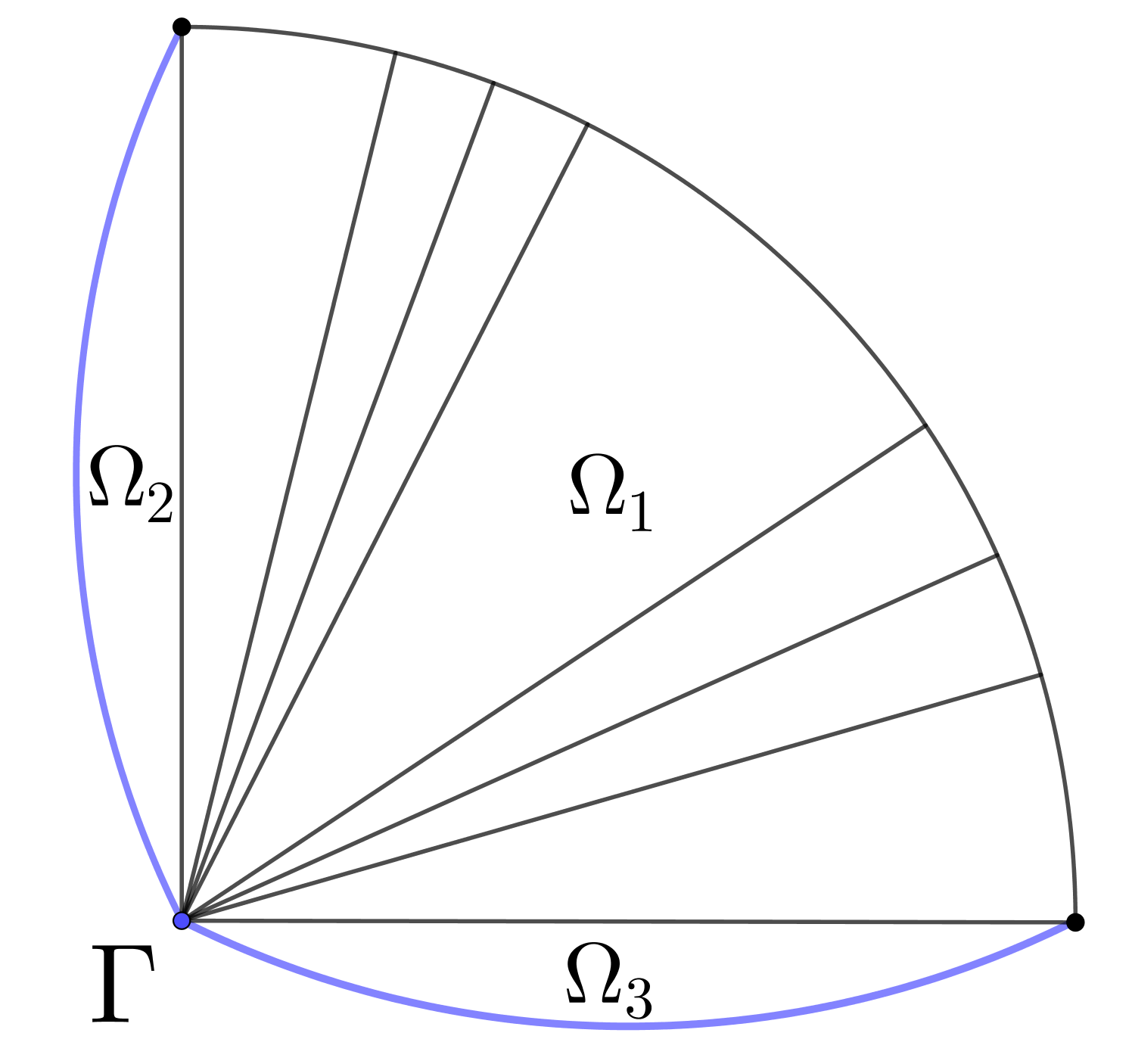} 
    \caption{Nonunique structure of solutions}
    \label{fig:niejmaa}
\end{figure}

We conclude this series of Examples by proving continuity of solutions in $\Omega \cup \Gamma$, when $\Gamma$ is connected and boundary data are continuous. In light of the analysis above, this is the optimal regularity result that we may obtain. Let us stress that the standard method of proving regularity used in the literature, using a comparison principle developed in \cite{JMN}, cannot be used here; in \cite{JMN} (see also \cite{FM}) continuity of solutions in low dimensions is obtained by extending a solution $u$ by a continuous function outside $\Omega$ and then applying \cite[Theorem 4.6]{JMN} to the superlevel sets $E_1 = \{ u \geq t \}$ and $E_2 = \{ u \geq s \}$. One of the hypotheses of \cite[Theorem 4.6]{JMN} is that $E_1 \backslash \Omega \subset \subset E_2 \backslash \Omega$; as we saw in Example \ref{ex:arc}, this hypothesis may fail even in the isotropic case when $\Gamma$ is a one-dimensional arc, see \ref{ex:arc}. A similar argument shows that we cannot apply the methods developed in \cite{JMN} to conclude uniqueness of solutions for continuous boundary data; nonetheless, using a different approach, some positive results about uniqueness of solutions in two dimensions in the isotropic case were proved in \cite{GRS2017NA}.

Nonetheless, under a bit more restrictive assumptions, we are able to prove continuity of solutions in $\Omega \cup \Gamma$ when $\Gamma$ is connected and $f$ is continuous. To this end, we will assume that a maximum principle for area-minimising surfaces holds, which is true for instance if $\phi(x,Du) = a(x)|Du|$ with sufficiently regular $a$. Before we give its precise statement, we first prove the regularity result or when $\phi$ is a strictly convex norm in two dimensions; then, connected $\phi$-minimal surfaces are line segments, which plays the role of the maximum principle. This proof will act as a model for the proof in the (higher-dimensional) weighted case.

\begin{theorem}\label{thm:continuitynorms}
Suppose that $\Omega \subset \mathbb{R}^2$ is an open bounded set with connected Lipschitz boundary. Let $\phi(x,\xi) = \| \xi \|$, where $\| \cdot \|$ is a strictly convex norm on $\mathbb{R}^2$. Suppose that $\Gamma \subset \partial\Omega$ is connected, i.e. it is an arc, and that $\Omega$ satisfies the barrier condition near $\Gamma$. Then, if $f \in C(\Gamma)$ and $u \in BV(\Omega)$ is a solution to problem \eqref{problem}, we have $u \in C(\Omega \cup \Gamma)$.
\end{theorem}

\begin{proof}
Given $x \in \Omega$, denote by $u^{\wedge}(x)$ and $u^{\vee}(x)$ the upper and lower limits of $u$ at $x$, namely
\begin{equation*}
u^\wedge(x) = \lim_{r \rightarrow 0} \mbox{ess sup}_{y \in B(x,r)} \, u(y), \qquad  u^\vee(x) = \lim_{r \rightarrow 0} \mbox{ess inf}_{y \in B(x,r)} \, u(y).
\end{equation*}
Suppose that $u$ is not continuous at $x_0 \in \Omega$; then, $u^\wedge(x_0) > u^\vee(x_0)$, and in particular there exist $s,t \in \mathbb{R}$ such that
\begin{equation*}
u^\vee(x_0) < t < s < u^\wedge(x_0),
\end{equation*}
in particular $x_0 \in \partial E_t \cap \partial E_s$. Take the connected components $\gamma$ of $E_t$ and $\gamma'$ of $E_s$ passing through $x_0$; because $E_s \subset E_t$ and $\gamma$ and $\gamma'$ are line segments, they have to coincide. By the Jordan curve theorem, $\gamma$ separates $\Omega$ into two open connected sets $\Omega^+$ and $\Omega^-$. Let $p_1, p_2 \in \partial\Omega$ be the endpoints of $\gamma$. Because $f \in C(\Gamma)$, using part (1) of Theorem \ref{thm:existencetracesense} we see that $p_1, p_2 \notin \Gamma$; therefore, $\Gamma$ intersects the boundary of exactly one of the sets $\Omega^\pm$ in $\mathbb{R}^2$ (in fact, it lies in it entirely); we choose the notation so that $\Gamma \subset \partial\Omega^+$. Now, take a function $v \in BV(\Omega)$ defined by the formula
\begin{equation*}
v(x) = \twopartdef{u(x)}{x \in \Omega^+}{\frac{t + s}{2}}{x \in \Omega^-;}
\end{equation*}
note that since $\Gamma \subset \partial\Omega^+$, it also satisfies the boundary condition. Denote by $u^\pm$ the trace of $u$ on $\gamma$ from $\Omega^\pm$. Then,
\begin{equation*}
\int_\Omega |Dv|_\phi = \int_{\Omega^+} |Du|_\phi + \int_{\gamma} \phi(x,\nu^\gamma) \, \bigg| u^+ - \frac{t+s}{2} \bigg| \, d\mathcal{H}^{1} + \int_{\Omega^-} 0 < \qquad\qquad\qquad\qquad\qquad\qquad\qquad 
\end{equation*}
\begin{equation*}
\qquad\qquad\qquad\qquad\qquad\qquad < \int_{\Omega^+} |Du|_\phi + \int_{\gamma} \phi(x, \nu^\gamma) \, |u^+ - u^-| \, d\mathcal{H}^{1} + \int_{\Omega^-} |Du|_\phi = \int_\Omega |Du|_\phi,
\end{equation*}
so $v$ has strictly lower total variation than $u$. Hence, $u$ was not a solution to problem \eqref{problem}, contradiction. Finally, let us note that continuity at points in $\Gamma$ follows from Theorem \ref{thm:existencetracesense}.
\end{proof}

We now give a precise statement of the maximum principle for area-minimising surfaces; in this form, it was proved in \cite[Theorem 3.1]{Zun} (for an isotropic version, see \cite{SWZ}).

\begin{proposition}\label{prop:maximumprinciple}
Assume that $\Omega \subset \mathbb{R}^N$ is an open bounded set with Lipschitz boundary. Let $\phi(x,\xi) = a(x) |\xi|$, where $a \in C^2(\overline{\Omega})$ is bounded from below by a positive number. Suppose that $E_1 \subset E_2$ are $\phi$-minimal sets in $\Omega$. If $x \in \partial E_1 \cap \partial E_2 \cap \Omega$, then $\partial E_1$ and $\partial E_2$ agree in some neighbourhood of $x$.
\end{proposition}

Now, we state the main regularity result, namely continuity of solutions in the weighted case when $\Gamma$ is connected. To this end, we will use Proposition \ref{prop:maximumprinciple}, but also some classical results in geometric measure theory for area-minimising integral currents, see \cite{Sim}. These originally refer to the isotropic case, but they have been improved to include the weighted case in \cite{JMN} and \cite{Zun}.

\begin{theorem}\label{thm:continuityweights}
Suppose that $\Omega \subset \mathbb{R}^N$ is an open bounded set with connected Lipschitz boundary. Let $\phi(x,\xi) = a(x) |\xi|$, where $a \in C^2(\overline{\Omega})$ is bounded from below by a positive number. Suppose that $\Gamma \subset \partial\Omega$ is connected and that $\Omega$ satisfies the barrier condition near $\Gamma$. Then, if $f \in C(\Gamma)$ and $u \in BV(\Omega)$ is a solution to problem \eqref{problem}, we have $u \in C(\Omega \cup \Gamma)$.
\end{theorem}

\begin{proof}
Suppose that $u$ is not continuous at $x_0 \in \Omega$; as in the proof of Theorem \ref{thm:continuitynorms}, there exist $s,t \in \mathbb{R}$ such that
\begin{equation*}
u^\vee(x_0) < t < s < u^\wedge(x_0)
\end{equation*}
and so $x_0 \in \partial E_t \cap \partial E_s$. Take the connected components $\gamma$ of $E_t$ and $\gamma'$ of $E_s$ passing through $x_0$; because $E_s \subset E_t$, Proposition \ref{prop:maximumprinciple} implies that $\gamma$ and $\gamma'$ coincide. 

We recall that by the classical estimates by Schoen, Simon and Almgren (see the extended discussion in \cite{JMN}), the set of singular points of $\phi$-area-minimising boundaries has Hausdorff dimension at most $N-3$. Hence, the set of regular points of $\gamma$ is dense in $\gamma$. Let $R = \mathrm{reg}(\gamma)$ denote the set of regular points of $\gamma$. Using an argument as in the proof of \cite[Lemma 4.5]{JMN} (this is a standard proof in geometric measure theory, for the isotropic case see \cite[Theorem 27.6]{Sim}), we see that there exists a set of finite perimeter $\Omega^+ \subset \Omega$ such that $\gamma = \overline{\mathrm{reg}(\gamma)} = \partial \Omega^+$ (here, we mean the boundary of $\Omega^+$ relative to $\Omega$). We set $\Omega^- = \Omega \backslash \Omega^+$.

Notice that since $\Gamma$ is connected, it lies entirely in the boundary (in $\mathbb{R}^N$) of exactly one of the sets $\Omega^\pm$. Suppose otherwise; then, for some $x \in \Gamma$ we have $x \in \partial\Omega^+ \cap \partial\Omega^- = \overline{\gamma}$, but this is not possible by part (1) of Theorem \ref{thm:existencetracesense} because $f \in C(\Gamma)$. We choose the notation so that $\Gamma \subset \partial\Omega^+$. Now, take a function $v \in BV(\Omega)$ defined by the formula
\begin{equation*}
v(x) = \twopartdef{u(x)}{x \in \Omega^+}{\frac{t + s}{2}}{x \in \Omega^-;}
\end{equation*}
note that since $\Gamma \subset \partial\Omega^+$, it also satisfies the boundary condition. Denote by $u^\pm$ the trace of $u$ on $\gamma$ from $\Omega^\pm$. Then,
\begin{equation*}
\int_\Omega |Dv|_\phi = \int_{\Omega^+} |Du|_\phi + \int_{\gamma} \phi(x,\nu^\gamma) \, \bigg| u^+ - \frac{t+s}{2} \bigg| \, d\mathcal{H}^{N-1} + \int_{\Omega^-} 0 < \qquad\qquad\qquad\qquad\qquad\qquad\qquad 
\end{equation*}
\begin{equation*}
\qquad\qquad\qquad\qquad\qquad\qquad < \int_{\Omega^+} |Du|_\phi + \int_{\gamma} \phi(x, \nu^\gamma) \, |u^+ - u^-| \, d\mathcal{H}^{N-1} + \int_{\Omega^-} |Du|_\phi = \int_\Omega |Du|_\phi,
\end{equation*}
so $v$ has strictly lower total variation than $u$. Hence, $u$ was not a solution to problem \eqref{problem}, contradiction. Finally, let us note that continuity at points in $\Gamma$ follows from Theorem \ref{thm:existencetracesense}.
\end{proof}

{\bf Acknowledgements.} This work was partly supported by the research project no. 2017/27/N/ST1/02418, ``Anisotropic least gradient problem'', funded by the National Science Centre, Poland.

\bibliographystyle{plain}

\end{document}